\pdfoutput=1
\RequirePackage{ifpdf}
\ifpdf 
\documentclass[pdftex]{sigma}
\else
\documentclass{sigma}
\fi

\usepackage{combelow}

\numberwithin{equation}{section}

\newtheorem{Theorem}{Theorem}[section]
\newtheorem*{Theorem*}{Theorem}

\newtheorem{Lemma}[Theorem]{Lemma}
\newtheorem{Proposition}[Theorem]{Proposition}
 { \theoremstyle{definition}
\newtheorem{Definition}[Theorem]{Definition}

\newtheorem{Example}[Theorem]{Example}

\newtheorem{Question}[Theorem]{Question}
 }

\def\bR{{\mathbb{R}}}

\newcommand{\calL}{{\mathcal L}}
\newcommand{\calM}{{\mathcal M}}
\newcommand{\calN}{{\mathcal N}}

\newcommand{\sbv}[2]{{\{{{#1},{#2}}\}}}
\newcommand{\courant}[2]{{[{{#1},{#2}}]_D}}

\newcommand{\liebra}[2]{{[{{#1},{#2}}]_{TM \oplus \bR}}}

\newcommand{\rhoe}{\rho_E}
\newcommand{\rhoa}{{\rho_A}}

\newcommand{\rhott}{\rho_{T \oplus T^*}}

\newcommand{\rhotr}{\rho_{TM \oplus \bR}}

\newcommand{\bracket}[2]{\langle #1,\,#2\rangle}

\newcommand{\inner}[2]{{({{#1},{#2}})}}

\newcommand{\omegac}{{\omega_{can}}}
\newcommand{\omegam}{{\omega_{\calM}}}

\newcommand{\Thetam}{{\Theta_{\calM}}}
\newcommand{\Thetan}{{\Theta_{\calN}}}

\newcommand{\rd}{\mathrm{d}}

\newcommand{\uu}{{\underline{u}}{}}
\newcommand{\uv}{{\underline{v}}{}}

\newcommand{\ualpha}{{\underline{\alpha}}{}}
\newcommand{\ubeta}{{\underline{\beta}}{}}

\newcommand{\ua}{{\underline{a}}{}}
\newcommand{\ub}{{\underline{b}}{}}
\newcommand{\umu}{\underline{\mu}}
\newcommand{\unablamu}{\underline{\nabla \mu}}

\newcommand{\baS}{{{}^A S}}

\newcommand{\wpi}{{\widetilde{\pi}}}

\newcommand{\bbT}{{\mathbb{T}}}

\newcommand{\nablabas}{{{}^A \nabla}^{\mathrm{bas}}}
\newcommand{\anabla}{{{}^A \nabla}}

\begin{document}

\newcommand{\arXivNumber}{2405.03533}

\renewcommand{\PaperNumber}{003}

\FirstPageHeading

\ShortArticleName{Comomentum Sections and Poisson Maps in Hamiltonian Lie Algebroids}

\ArticleName{Comomentum Sections and Poisson Maps\\ in Hamiltonian Lie Algebroids}

\Author{Yuji HIROTA~$^{\rm a}$ and Noriaki IKEDA~$^{\rm b}$}

\AuthorNameForHeading{Y.~Hirota and N.~Ikeda}

\Address{$^{\rm a)}$~Division of Integrated Science, Azabu University, Sagamihara, Kanagawa 252-5201, Japan}
\EmailD{\href{mailto:hirota@azabu-u.ac.jp}{hirota@azabu-u.ac.jp}}

\Address{$^{\rm b)}$~Department of Mathematical Sciences, Ritsumeikan University, \\
\hphantom{$^{\rm b)}$}~Kusatsu, Shiga 525-8577, Japan}
\EmailD{\href{mailto:nikeda@se.ritsumei.ac.jp}{nikeda@se.ritsumei.ac.jp}}
\URLaddressD{\url{http://www.ritsumei.ac.jp/~nikeda/}}

\ArticleDates{Received June 16, 2024, in final form December 29, 2024; Published online January 05, 2025}

\Abstract{In a Hamiltonian Lie algebroid over a pre-symplectic manifold and over a Poisson manifold, we introduce a map corresponding to a comomentum map, called a comomentum section. We show that the comomentum section gives a Lie algebroid morphism among Lie algebroids. Moreover, we prove that a momentum section on a Hamiltonian Lie algebroid is a~Poisson map between proper Poisson manifolds, which is a generalization that a momentum map is a Poisson map between the symplectic manifold to dual of the Lie algebra. Finally, a momentum section is reinterpreted as a Dirac morphism on Dirac structures.}


\Classification{53D17; 53D20; 53D05}

\section{Introduction}

A momentum map is a fundamental object in symplectic geometry
defined on a symplectic manifold with a Lie group action.
Then the action is a Hamiltonian action and the total space is called a Hamiltonian $G$-space.

Analyses of physical models suggest that Lie group actions in momentum maps should be generalized to `Lie groupoid actions' to realize symmetries and conserved quantities in physical theories \cite{Alekseev:2004np, Blohmann:2010jd, Cattaneo:2000iw}.
A `groupoid' generalization of a Lie algebra is a Lie algebroid.
A Lie algebroid is an infinitesimal object of a Lie groupoid analogous to the way that a Lie algebra is an infinitesimal object of a Lie group.

Recently, a generalization of a momentum map and a Hamiltonian $G$-space
over a pre-symplectic manifold
has been proposed in a Lie algebroid (Lie groupoid) setting \cite{Blohmann:2018}.
It is inspired by the analysis of the Hamiltonian formalism of general relativity \cite{Blohmann:2010jd} and compatibility of physical models with Lie algebroid structures \cite{Kotov:2016lpx}.

Mathematically, the idea is natural in the following sense. A momentum map $\mu$ is a map from a smooth manifold $M$
to dual of a Lie algebra $\mathfrak{g}^*$,
$\mu\colon M \rightarrow \mathfrak{g}^*$.
It is also regarded as a section of a trivial vector bundle $M \times \mathfrak{g}^*$.
A momentum map on the trivial bundle has been generalized
to a section on the dual of a vector bundle $A$, $\mu \in \Gamma(A^*)$,
satisfying certain consistency conditions.
The generalization $\mu \in \Gamma(A^*)$ is called a momentum section,
and the Hamiltonian $G$-space is generalized
to a Hamiltonian Lie algebroid~\cite{Blohmann:2018}.

Further generalizations have been analyzed.
A momentum section and a Hamiltonian Lie algebroid over a Poisson manifold has also been proposed in \cite{Blohmann:2023}.
A momentum section has been generalized to a momentum section on a Courant algebroid \cite{Ikeda:2021fjk}, over a pre-multisymplectic mani\-fold \cite{Hirota:2021isx}, over bundle-valued (multi)symplectic structures \cite{Hirota:2023xqd} and over a Dirac structure~\cite{Ikeda:2023pdr}.
In our paper, based on two papers \cite{Blohmann:2023,Blohmann:2018},
we consider both momentum sections over \mbox{a~(pre-)symplectic} manifold and over a~Poisson manifold.

Momentum maps have several important and elegant properties.
If there exists a momentum map on a symplectic or a Poisson manifold, we have symplectic \cite{MarsdenWeinstein, Meyer} or Poisson reductions~\cite{MarsdenRatiu}.
One essential property to make reductions consistent is that the momentum map is a~Poisson map from $M$ to $\mathfrak{g}^*$.
In other words, comomentum map $\mu^*\colon \mathfrak{g} \rightarrow C^{\infty}(M)$
is a Lie algebra morphism.
However, a momentum section is not necessarily a Poisson map from $M$ to~$A^*$~\cite{Blohmann:2023}.
Our motivation is to improve this problem to consider reductions for the Hamiltonian Lie algebroid setting.
One idea to make a momentum section a Poisson map is to impose a condition compatible with a~Poisson structure and a Lie algebroid $A$ \cite{Hirota:2023xqd}.
We propose another idea to construct a Poisson map in a Hamiltonian Lie algebroid.

In this paper, we generalize two properties of momentum maps to momentum sections. One is a comomentum map and another is the momentum map as a Poisson map.

For a momentum map $\mu$, we can define a \textit{comomentum map}
$\mu^*\colon \mathfrak{g} \rightarrow C^{\infty}(M)$ as a dual map,
which has properties induced from the momentum map.
In particular, a comomentum map has the following property.

\begin{Proposition}
The comomentum map is a Lie algebra morphism from $\mathfrak{g}$ to $C^{\infty}(M)$, where a Lie algebra structure on $C^{\infty}(M)$ is defined by the Poisson bracket.
\end{Proposition}
Here, for two Lie algebras $\mathfrak{g}_1$ and $\mathfrak{g}_2$,
a Lie algebra morphism is a map $\phi\colon\mathfrak{g}_1 \rightarrow \mathfrak{g}_2$ satisfying~${\phi([e_1, e_2]_1) = [\phi(e_1), \phi(e_2)]_2}$
for every $e_1, e_2 \in \mathfrak{g}_1$.
A Lie algebra structure on $C^{\infty}(M)$ is given by the Poisson bracket.

We have the following natural question.

\begin{Question}
Construct a comomentum section such that it is a Lie algebroid morphism
between suitable two Lie algebroids.
\end{Question}

In this paper, we define and analyze a bracket-compatible \textit{comomentum section} $\mu^*\colon\Gamma(A) \rightarrow C^{\infty}(M)$ corresponding to a bracket-compatible momentum section $\mu$.
It is proved that the comomentum section $\mu^*$ is a \textit{Lie algebroid morphism} from $A$ to a proper space including $C^{\infty}(M)$.

\begin{Definition}
Assume that $(A_1, [-,-]_1, \rhoa_1)$ and $(A_2, [-,-]_2, \rhoa_2)$ are
two Lie algebroids over $M$.
A \textit{Lie algebroid morphism} between two Lie algebroids $A_1$ and $A_2$
is a vector bundle morphism $\phi\colon A_1 \rightarrow A_2$
such that
\begin{gather*}
\phi([e_1, e_2]_1) = [\phi(e_1), \phi(e_2)]_2,\nonumber
\\
\rhoa_2 \circ \phi = \rhoa_1
\end{gather*}
for $e_1, e_2 \in \Gamma(A_1)$.
\end{Definition}

The idea on a pre-symplectic manifold
is that we consider a pair given by the anchor map $\rho$
of the Lie algebroid $A$ and comomentum section $\mu^*$.
Since $\rho$ is a map from $A$ to $TM$,
$\rho + \mu^*$ is regarded as a map from $A$ to $TM \oplus \bR$.
On a Poisson manifold, $\rho$ is replaced to
$(\nabla \mu)^*\colon A \rightarrow T^*M$.
Then $(\nabla \mu)^* + \mu^*$ is a Lie algebroid morphism
from $A$ to $T^*M \oplus \bR$ if we define a Lie algebroid structure on
$T^*M \oplus \bR$.

Another important property of a momentum map is that it is a Poisson map
between two Poisson manifolds $M$ and $\mathfrak{g}^*$.

The dual of a Lie algebra $\mathfrak{g}^*$ has the so called
Kirillov--Kostant--Souriau (KKS) Poisson structure
$\pi_{\rm KKS} \in \wedge^2(T\mathfrak{g}^*)$ \cite{Kirillov}.
A momentum map $\mu$ is a map between two Poisson manifolds~$(M, \pi)$ and $(\mathfrak{g}^*, \pi_{\rm KKS})$ satisfying the following property.
\begin{Proposition}\label{MMPM}
A momentum map $\mu\colon M \rightarrow \mathfrak{g}^*$ is a Poisson map.
\end{Proposition}
The purpose of this paper is to generalize this proposition to a momentum section.
Here a~Poisson map is defined as follows.
\begin{Definition}
Let $(M_1, \pi_1)$ and $(M_2, \pi_2)$ be Poisson manifolds.
A smooth map $\psi$ is a \textit{Poisson map} if
$\psi^*\colon C^{\infty}(M_2) \rightarrow C^{\infty}(M_1)$ satisfies
$
\psi^* \sbv{f}{g}_2 = \sbv{\psi^* f}{\psi^* g}_1$,
for $f, g \in C^{\infty}(M_2)$, where~${\sbv{-}{-}_1}$
and $\sbv{-}{-}_2$ are Poisson brackets on $M_1$ and $M_2$.
\end{Definition}

A momentum section $\mu \in \Gamma(A^*)$ is regarded as a map $\mu\colon M \rightarrow A^*$.
A generalization of Proposition \ref{MMPM} for a bracket-compatible momentum section has been analyzed in \cite{Blohmann:2023}.
On the dual of the Lie algebroid $A^*$, a bivector field
$\wpi_{A^*} = \wpi + \pi_{A^*}$ is defined, where $\wpi$
is the lift of the Poisson bivector field $\pi$ on $M$ to $A^*$
and $\pi_{A^*}$ is a Poisson bivector field induced from a~Lie algebroid structure on $A$.
In the Hamiltonian Lie algebroid over a Poisson manifold~$M$,
${\mu\colon M \rightarrow A^*}$ is a~bivector map
if the basic curvature vanishes, ${}^A S =0$.
Here a bivector map~${\psi\colon M \rightarrow A^*}$ is a bilinear map such that
$
\sbv{\psi^* a}{\psi^* b}_M = \psi^* \sbv{a}{b}_{A^*}$,
for every $a, b \in C^{\infty}(A^*)$,
where a bilinear bracket
$\sbv{-}{-}_M$ is the Poisson bracket induced from~$\pi_M$
and $\sbv{-}{-}_{A^*}$ is a bilinear bracket induced from $\wpi_{A^*}$.
In this result, $\wpi_{A^*}$ is not necessarily a Poisson bivector field on $A^*$, i.e., $\sbv{-}{-}_{A^*}$ does not satisfy the Jacobi identity, which means that $\mu$ is not necessarily a Poisson map.\looseness=-1

\begin{Question}
Is a momentum section $\mu$ regarded as a Poisson map of two proper Poisson manifolds?
\end{Question}

We prove that a bracket-compatible momentum section is a
Poisson map from $T^*M \oplus \bR$ to $A^*$ under Poisson structures
induced from the Poisson structure $\pi_M$ on $M$ and
the Poisson structure $\pi_{A^*}$ induced from
the Lie algebroid structure on $A$.

This paper is organized as follows.
Section \ref{sec2} is the preparation.
In Section \ref{sec2}, after a~Lie algebroid, connections
and related notion are introduced,
a Hamiltonian Lie algebroid and a~momentum section over a~pre-symplectic manifold and a Poisson manifold are explained.
Moreover, a~Courant algebroid and a~Dirac structure are introduced.
In Section \ref{sec3}, a relation of a momentum section
with the basic curvature is discussed and some formulas are given.
In Section \ref{sec:CMSLAM}, comomentum sections are defined for
a pre-symplectic case and a Poisson case, and Lie algebroid morphisms
induced from momentum sections are constructed.
In Section \ref{sec5}, a Poisson map induced from a momentum section is constructed.
In Section \ref{sec:Dirac}, a momentum section is reinterpreted as a Dirac morphism.

\section{Preliminary}\label{sec2}

In this section, we summarize definitions and previous results,
including a Lie algebroid and connections, momentum sections and Hamiltonian Lie algebroids over a pre-symplectic manifold and over a Poisson manifold,
as well as a Courant algebroid and a Dirac structure.

\subsection{Lie algebroids}

\begin{Definition}
Let $A$ be a vector bundle over a smooth manifold $M$.
A Lie algebroid $(A, [-,-], \rho=\rhoa)$ is a vector bundle $A$ with
a bundle map $\rho=\rhoa\colon A \rightarrow TM$ called the anchor map,
and a Lie bracket
$[-,-]= [-,-]_A\colon \Gamma(A) \times \Gamma(A) \rightarrow \Gamma(A)$
satisfying the Leibniz rule,
$
[e_1, fe_2] = f [e_1, e_2] + \rhoa(e_1) f \cdot e_2$,
where $e_1, e_2 \in \Gamma(A)$ and $f \in C^{\infty}(M)$.
\end{Definition}
A Lie algebroid is a generalization of a Lie algebra and the space of vector fields on a smooth manifold.
\begin{Example}[Lie algebras]
Let a manifold $M$ be one point $M = \{\mathrm{pt} \}$.
Then a Lie algebroid is a Lie algebra $\mathfrak{g}$.
\end{Example}
\begin{Example}[tangent Lie algebroids]\label{tangentLA}
Let a vector bundle $A$ be a tangent bundle $TM$, $\rhoa = \mathrm{id}$
and a bracket $[-,-]$ is a normal Lie bracket
on the space of vector fields $\mathfrak{X}(M)$.
$TM$ together with $\rho$ and $[-,-]$ is a Lie algebroid.
It is called a \textit{tangent Lie algebroid}.
\end{Example}

\begin{Example}[action Lie algebroids]\label{actionLA}
Assume a smooth action of a Lie group $G$ to a smooth manifold $M$,
{$M \times G \rightarrow M$.}
The differential map induces an infinitesimal action on the manifold~$M$ of the Lie algebra $\mathfrak{g}$ of $G$.
Since $\mathfrak{g}$ acts as a differential operator on $M$,
the differential map
is a bundle map $\rho\colon M \times \mathfrak{g} \rightarrow TM$.
Consistency of a Lie bracket requires that $\rho$ is
a Lie algebra morphism such that
\begin{gather}
~[\rho(e_1), \rho(e_2)] = \rho([e_1, e_2]),
\label{almostLA}
\end{gather}
where the bracket on the left-hand side
of \eqref{almostLA} is the Lie bracket of vector fields.
These data give a Lie algebroid $(A= M \times \mathfrak{g}, [-,-], \rho)$.
This Lie algebroid is called an \textit{action Lie algebroid}.
\end{Example}

\begin{Example}[Poisson Lie algebroids]\label{Poisson}
A bivector field $\pi \in \Gamma\bigl(\wedge^2 TM\bigr)$ is called a Poisson bivector field if $[\pi, \pi]_S =0$, where $[-,-]_S$ is the Schouten bracket on the space of multivector fields, $\Gamma(\wedge^{\bullet} TM)$.
A smooth manifold $M$ with a Poisson bivector field $\pi$ is called a Poisson manifold and denoted by $(M, \pi)$.

Let $(M, \pi)$ be a Poisson manifold. Then a Lie algebroid structure is induced on $T^*M$.
A~bundle map is defined as
$\pi^{\sharp}\colon T^*M \rightarrow TM$ by $\bigl\langle\pi^{\sharp}(\alpha),\beta\bigr\rangle
= \pi(\alpha, \beta)$ for all $\beta \in \Omega^1(M)$.
$\rho= - \pi^{\sharp}$ is the anchor map, and
a Lie bracket on $\Omega^1(M)$ is defined by the Koszul bracket
\begin{gather*}
[\alpha, \beta]_{\pi} = \calL_{\pi^{\sharp} (\alpha)}\beta - \calL_{\pi^{\sharp} (\beta)} \alpha - \rd(\pi(\alpha, \beta)),
\end{gather*}
where $\alpha, \beta \in \Omega^1(M)$.
$\bigl(T^*M, [-, -]_{\pi}, -\pi^{\sharp}\bigr)$ is a Lie algebroid.
\end{Example}

One can refer to reviews and textbooks for basic properties of Lie algebroids, see, for instance,~\cite{Mackenzie}.

For a Lie algebroid $A$, sections of the exterior algebra of $A^*$ are called \textit{$A$-differential forms}.
A~differential ${}^A \rd\colon \Gamma(\wedge^m A^*)
\rightarrow \Gamma\bigl(\wedge^{m+1} A^*\bigr)$ on the spaces of $A$-differential forms, $\Gamma(\wedge^{\bullet} A^*)$,
called a \textit{Lie algebroid differential}, or an \textit{$A$-differential},
is defined as follows.
\begin{Definition}
For an $A$-differential form $\eta \in \Gamma(\wedge^m A^*)$,
a Lie algebroid differential ${}^A \rd\colon \allowbreak\Gamma(\wedge^m A^*)
\rightarrow \Gamma\bigl(\wedge^{m+1} A^*\bigr)$ is defined by
\begin{align*}
{}^A \rd \eta(e_1, \ldots, e_{m+1})
={}& \sum_{i=1}^{m+1} (-1)^{i-1} \rhoa(e_i) \eta(e_1, \ldots,
\check{e_i}, \ldots, e_{m+1})
 \\ &
+ \sum_{1 \leq i < j \leq m+1} (-1)^{i+j} \eta([e_i, e_j], e_1, \ldots, \check{e_i}, \ldots, \check{e_j}, \ldots, e_{m+1}),
\end{align*}
where $e_i \in \Gamma(A)$, and the ha\'{c}\v{e}k symbol signifies omitting that argument.
\end{Definition}
The $A$-differential satisfies $\bigl({}^A \rd\bigr)^2=0$.
It is a generalization of the de Rham differential on~$T^*M$ and the Chevalley--Eilenberg differential on a Lie algebra.

\begin{Definition}
Assume that $(A_1, [-,-]_1, \rhoa_1)$ and $(A_2, [-,-]_2, \rhoa_2)$ are
two Lie algebroids over $M$.
A \textit{Lie algebroid morphism} between two Lie algebroids $A_1$ and $A_2$
is a vector bundle morphism $\phi\colon A_1 \rightarrow A_2$
such that
\begin{gather}
\phi([e_1, e_2]_1) = [\phi(e_1), \phi(e_2)]_2,\nonumber
\\
\rhoa_2 \circ \phi = \rhoa_1
\label{LAmorphism02}
\end{gather}
for $e_1, e_2 \in \Gamma(A_1)$.\footnote{For Lie algebroids over different base manifolds, we can define more general Lie algebroid morphism \cite{Mackenzie}.
For two Lie algebroids $(A_1, M_1)$ and $(A_2, M_2)$, a morphism $\phi\colon A_1 \rightarrow A_2$ is a vector bundle morphism whose graph
$\mathrm{Gr}(\phi) \subset A_1 \times A_2$ is a Lie subalgebroid of $A_1 \times A_2$.}
\end{Definition}

\subsection{Connections on Lie algebroids}\label{connectionLA}
We introduce several connections on a vector bundle $E$.
\begin{Definition}
A connection is an $\bR$-linear map,
$\nabla\colon\Gamma(E)\rightarrow \Gamma(E \otimes T^*M)$,
satisfying the Leibniz rule,
$
\nabla (f s) = f \nabla s + (\rd f) \otimes s$,
for $s \in \Gamma(E)$ and $f \in C^{\infty}(M)$.
A dual connection on $E^*$ is defined by the equation
\[
\rd \inner{\mu}{s} = \bracket{\nabla \mu}{s} + \bracket{\mu}{\nabla s}
\]
for all sections $\mu \in \Gamma(E^*)$ and $s \in \Gamma(E)$,
where $\bracket{-}{-}$ is the pairing between $E$ and $E^*$.
\end{Definition}
We use the same notation $\nabla$ for the dual connection.

On a Lie algebroid, another derivation called an $A$-connection is defined.
\begin{Definition}
Let $A$ be a Lie algebroid over a smooth manifold $M$ and $E$ be a vector bundle over the same base manifold $M$.
An \textit{$A$-connection} on a vector bundle $E$
with respect to the Lie algebroid $A$ is a $\bR$-linear
map,
$\anabla\colon \Gamma(E) \rightarrow \Gamma(E \otimes A^*)$,
satisfying
\[
\anabla_e (f s) = f \anabla_e s + (\rhoa(e) f) s
\]
for $e \in \Gamma(A)$, $s \in \Gamma(E)$ and $f \in C^{\infty}(M)$.
\end{Definition}
The ordinary connection is regarded as an
$A$-connection for $A=TM$, $\nabla = {}^{TM} \nabla$.

If an ordinary connection $\nabla$ on $A$ as a vector bundle is given,
an $A$-connection on $A$ is simply given by
\smash{$
 \anabla_{e} e^{\prime} := \nabla_{\rhoa(e^{\prime})} {e}$},
for $e, e^{\prime} \in \Gamma(A)$.

Another $A$-connection called the \textit{basic $A$-connection} on
the tangent bundle $E=TM$,
$\smash{\nablabas}\colon\allowbreak \Gamma(TM) \rightarrow \Gamma(TM \otimes A^*)$
is defined.
\begin{Definition}
The \textit{basic $A$-connection} on $TM$,
\smash{$\nablabas\colon \Gamma(TM) \rightarrow \Gamma(TM \otimes A^*)$}
is defined~by
\begin{gather}
\nablabas_{e} v := \calL_{\rhoa(e)} v + \rhoa(\nabla_v e)
= [\rhoa(e), v] + \rhoa(\nabla_v e),
\label{stEconnection1}
\end{gather}
where
$e \in \Gamma(A)$ and $v \in \mathfrak{X}(M)$.
\end{Definition}
For a $1$-form $\alpha \in \Omega^1(M)$, the basic $A$-connection
is given by
\begin{gather}
\nablabas_{e} \alpha := \calL_{\rhoa(e)} \alpha
+ \bracket{\rhoa(\nabla e)}{\alpha}.
\label{Econoneform}
\end{gather}
Throughout this paper, $A$-connections $\anabla$ on $TM$ and $T^*M$ are always
the basic $A$-connection \smash{$\anabla = \nablabas$},
\eqref{stEconnection1} and
\eqref{Econoneform}.

Given a connection $\nabla$, a covariant derivative is generalized
to the derivation on the space of differential forms taking a value on
$\wedge^m A^*$,
$\Omega^k(M, \wedge^m A^*)$ called an \textit{exterior covariant derivative}.
Similarly, an $A$-connection $\anabla$ can be generalized to the derivation
satisfying the Leibniz rule for sections on $\wedge^m A^*$.
It is called \textit{the $A$-exterior covariant derivative} $\anabla$.
They are denoted by the same notation $\nabla$ and ${}^A\nabla$.

Let $\Omega^k(M, \wedge^m A^*) $ be the space of $k$-forms taking values on
$\wedge^m A^*$.
\begin{Definition}
For $\Omega^k(M, \wedge^m A) = \Gamma\bigl(\wedge^m A^* \otimes \wedge^k T^*M\bigr)$,
the \textit{$A$-exterior covariant derivative}
$\anabla\colon \Omega^k(M, \wedge^m A^*) \rightarrow \Omega^k\bigl(M, \wedge^{m+1} A^*\bigr)$ is defined by
\begin{align*}
\bigl(\anabla \alpha\bigr)(e_1, \ldots, e_{m+1})
:={}& \sum_{i=1}^{m+1} (-1)^{i-1}
\anabla_{e_i}
(\alpha(e_1, \ldots, \check{e_i}, \ldots, e_{m+1}))
\nonumber \\ &
+ \sum_{1 \leq i < j \leq m+1} (-1)^{i+j} \alpha([e_i, e_j], e_1, \ldots, \check{e_i}, \ldots, \check{e_j}, \ldots, e_{m+1})
\end{align*}
for $\alpha \in \Omega^k(M, \wedge^m A^*)$ and $e_i \in \Gamma(A)$.
\end{Definition}
Note that the $A$-exterior covariant derivative increases the order of $\wedge^m A^*$.

The $A$-exterior covariant derivative for $E=TM$ is given by
taking the pairing,
\begin{gather*}
{}^A \rd \bracket{\phi}{\alpha} =
\bigl\langle\bigl(\anabla \phi\bigr),\alpha\bigr\rangle
+ \bigl\langle\phi,\bigl(\anabla \alpha\bigr)\bigr\rangle
\end{gather*}
for $\phi \in \mathfrak{X}^k(M, \wedge^m A^*)$
and $\alpha_i \in \Omega^k(M)$.

One can refer to \cite{AbadCrainic, CrainicFernandes, DufourZung}
about a theory of general and basic $A$-connections on a Lie algebroid.

For two connections $\nabla$ and $\anabla$, various torsions and curvatures are introduced.
Additional to the normal curvature $R \in \Omega^2(M, A \otimes A^*)$ and the torsion for a vector bundle connection $\nabla$,
similar quantities for the $A$-connections are introduced.
An \textit{$A$-curvature} ${}^A R \in \Gamma\bigl(\wedge^2 A^* \otimes A \otimes A^*\bigr)$ is defined by
\[
{}^A R(e, e^{\prime}) := \bigl[{}^A\nabla_e, {}^A\nabla_{e^{\prime}}\bigr] - {}^A\nabla_{[e, e^{\prime}]}
\]
for $e, e^{\prime} \in \Gamma(A)$. It does not appear explicitly in our paper.
Important geometric quantities are an ordinary curvature, an $A$-torsion
and a basic curvature \cite{Blaom}.
\begin{Definition}
An ordinary \textit{curvature} $R \in \Omega^2(M, A \otimes A^*)$,
an \textit{$A$-torsion}, ${}^A T \in \Gamma\bigl(A \otimes \wedge^2 A^*\bigr)$,
and a \textit{basic curvature}, $\baS \in \Omega^1\bigl(M, \wedge^2 A^* \otimes A\bigr)$, are defined by
\begin{gather}
R(v, v^{\prime}) := [\nabla_v, \nabla_{v^{\prime}}] - \nabla_{[v, v^{\prime}]},\nonumber
\\
{}^A T(e, e^{\prime}) := \nabla_{\rho(e)} e^{\prime} - \nabla_{\rho(e^{\prime})} e
- [e, e^{\prime}],
\label{Etorsion}
\\
\baS(e, e^{\prime}) :=
[e, \nabla e^{\prime}] - [e^{\prime}, \nabla e]
- \nabla[e, e^{\prime}]
- \nabla_{\rhoa(\nabla e)} e^{\prime} + \nabla_{\rhoa(\nabla e^{\prime})} e\nonumber\\
\phantom{\baS(e, e^{\prime})}{}= (\nabla {}^A T + 2 \mathrm{Alt} \, \iota_\rho R)(e, e^{\prime})
\nonumber
\end{gather}
for $v, v^{\prime} \in \mathfrak{X}(M)$ and $e, e^{\prime} \in \Gamma(A)$.
Note that $\iota_\rho R \in \Omega^1(M, A^* \otimes A^* \otimes A)$ since
$\rho \in \Gamma(A^* \otimes TM)$ and $\iota_\rho$ gives the contraction
between $TM$ and $T^*M$. Notation $\mathrm{Alt}$ means skew symmetrization on $A^* \otimes A^*$, which gives an element in $\Omega^1\bigl(M, \wedge^2 A^* \otimes A\bigr)$.
\end{Definition}

\subsection{Momentum sections and Hamiltonian Lie algebroids}\label{MSHLA}
In this section, a bracket-compatible momentum section and a Hamiltonian Lie
algebroid, which are a generalization of a momentum map on a symplectic
manifold, are reviewed \cite{Blohmann:2023,Blohmann:2018}.

A closed $2$-form $\omega \in \Omega^2(M)$ on $M$ is called a pre-symplectic form. A pair $(M, \omega)$ of a manifold~$M$ and a pre-symplectic form $\omega$ is called a pre-symplectic manifold. If $\omega$ is nondegenerate, $(M, \omega)$ is a symplectic manifold.

On a pre-symplectic manifold $M$, the following three conditions are introduced.

\begin{Definition}[momentum sections over pre-symplectic manifolds]\label{momsymp}
Suppose that a base manifold $(M, \omega)$ is a pre-symplectic manifold, and
take a Lie algebroid $(A, [-,-], \rhoa)$ over $M$.
\begin{itemize}\itemsep=0pt
\item[(S1)] A Lie algebroid $A$ is called \textit{pre-symplectically anchored} if $\omega$ satisfies
\begin{gather}
\anabla \omega =0.\label{HH1}
\end{gather}
\item[(S2)] A section $\mu \in \Gamma(A^*)$ is a $\nabla$-\textit{momentum section} if it satisfies\footnote{Notation of the momentum section such as $\mu(e)$, $(\nabla \mu)(e)$, etc.~are in fact pairings of $A^*$ and $A$, $\bracket{\mu}{e}$, $\bracket{\nabla \mu}{e}$, etc.}
\begin{gather}
(\nabla \mu)(e) = - \iota_{\rhoa(e)} \omega\label{HH2}
\end{gather}
for $e \in \Gamma(A)$.
\item[(S3)] $\mu$ is \textit{bracket-compatible} if it satisfies
\begin{gather}
\bigl({}^A \rd \mu\bigr)(e_1, e_2) = \omega(\rhoa(e_1), \rhoa(e_2)),
\label{HH3}
\end{gather}
where $e_1, e_2 \in \Gamma(A)$.
\end{itemize}
\end{Definition}
Note that the above definition depends on choice of a connection $\nabla$.

\begin{Definition}\label{HamiltonianLAS}
A Lie algebroid $A$ over a pre-symplectic manifold with a connection $\nabla$ and a~section $\mu \in \Gamma(A^*)$ is called \textit{Hamiltonian}\footnote{If the condition is satisfied on a neighborhood of every point in $M$, it is called \textit{locally Hamiltonian} \cite{Blohmann:2018}. All the analysis in this paper are applicable in the locally Hamiltonian case.}
if equations~\eqref{HH1}, \eqref{HH2} and \eqref{HH3}
are satisfied.
\end{Definition}
On a trivial bundle, a momentum section is equivalent to a momentum map.
Suppose that~$M$ has an action of a Lie group $G$
and $\omega$ is a symplectic form.
For a Lie algebra $\mathfrak{g}$ of $G$, a trivial bundle
$A = M \times \mathfrak{g}$ has an action Lie algebroid structure
in Example \ref{actionLA}.
A section $e \in \Gamma(M \times \mathfrak{g})$ is
restricted to the constant section, which is identified
to an element of $\mathfrak{g}$.
We can take a trivial connection $\nabla=\rd$ on the trivial bundle
$M \times \mathfrak{g}$.
Then conditions of Definition \ref{momsymp} reduce
to the following conditions.
Equation \eqref{HH2} is
\begin{gather}
(\rd \mu)(e) = - \iota_{\rhoa(e)} \omega,
\label{MM2}
\end{gather}
where $e$ is a constant section.
Equation~\eqref{MM2} means that $\mu(e)$ is the Hamiltonian function for the Lie algebra action $\rhoa(e)$.
Equation~\eqref{HH1} is
$\anabla_e \omega = \calL_{\rhoa(e)} \omega = 0$ from the definition of
the $A$-connection.
This equation is trivially satisfied from $\rd \omega =0$ and equation~\eqref{MM2}.
Equation~\eqref{HH3} is equivalent to
\[
\rhoa(e_1) \mu(e_2) = \mu([e_1, e_2])
\]
under \eqref{HH2} and \eqref{HH1}, which means that $\mu$ is infinitesimally
equivariant.
Since the section $\mu \in \Gamma(M \times \mathfrak{g}^*)$ is a map
$\mu\colon M \rightarrow \mathfrak{g}^*$,
therefore, $\mu$ is a momentum map on the pre-symplectic manifold $M$.

A Hamiltonian Lie algebroid over a Poisson manifold is defined as follows. \cite{Blohmann:2023}
\begin{Definition}[momentum sections over Poisson manifolds]\label{momPois}
Let $(M, \pi)$ be a Poisson manifold with a Poisson bivector field $\pi \in \Gamma\bigl(\wedge^2 TM\bigr)$
and $(A, [-,-], \rhoa)$ be a Lie algebroid over $M$.
\begin{itemize}\itemsep=0pt
\item[(P1)] $A$ is called \textit{Poisson anchored} if $\pi$ satisfies
\begin{gather}
\anabla \pi = 0.
\label{PMS1}
\end{gather}
\item[(P2)] A section $\mu \in \Gamma(A^*)$ is a $\nabla$-\textit{momentum section} if it satisfies
\begin{gather}
\rhoa(e) = - \pi^{\sharp} ((\nabla \mu)(e))\label{PMS2}
\end{gather}
for $e \in \Gamma(A)$.
\item[(P3)] $\mu$ is called \textit{bracket-compatible} if it satisfies
\begin{gather}
\bigl({}^A \rd \mu\bigr) (e_1, e_2) = - \pi((\nabla \mu)(e_1), (\nabla \mu)(e_2))\label{PMS3}
\end{gather}
for $e_1, e_2 \in \Gamma(A)$.
\end{itemize}
\end{Definition}
In Definition \ref{momPois}, conditions depend on choice of
a connection $\nabla$.

\begin{Definition}\label{HamiltonianLAP}
A Lie algebroid $A$ over a Poisson manifold with a connection $\nabla$ and a section~${\mu \in \Gamma(A^*)}$ is called \textit{Hamiltonian}
if equations~\eqref{PMS1}, \eqref{PMS2} and \eqref{PMS3}
are satisfied.
\end{Definition}
If $\pi$ is nondegenerate, $M$ is a symplectic manifold with
$\omega = \pi^{-1}$.
A Hamiltonian Lie algebroid over a nondegenerate Poisson manifold
is a Hamiltonian Lie algebroid over a symplectic manifold
as per Definition \ref{HamiltonianLAS}.

\subsection{Courant algebroids and Dirac structures}
In this subsection, a Courant algebroid and a Dirac structure \cite{Courant, LWX} are introduced
as preparations for the following sections.

\begin{Definition}\label{courantdefinition}
A Courant algebroid is a vector bundle $E$ over $M$,
which has a nondegenerate symmetric bilinear form
$\bracket{-}{-}$, a bilinear operation $\courant{-}{-}$ on $\Gamma(E)$,
and a bundle map called an anchor map,
$\rhoe\colon E \longrightarrow TM$, satisfying the following properties:
\begin{itemize}\itemsep=0pt
\item[(1)] $\courant{e_1}{\courant{e_2}{e_3}} = \courant{\courant{e_1}{e_2}}{e_3} + \courant{e_2}{\courant{e_1}{e_3}}$,
\item[(2)] $\rhoe(\courant{e_1}{e_2}) = [\rhoe(e_1), \rhoe(e_2)]$,
\item[(3)] $\courant{e_1}{f e_2} = f \courant{e_1}{e_2}
+ (\rhoe(e_1)f)e_2$,
\item[(4)] $\courant{e}{e} = \frac{1}{2} {\cal D} \bracket{e}{e}$,
\item[(5)] $\rhoe(e_1) \bracket{e_2}{e_3}
= \bracket{\courant{e_1}{e_2}}{e_3} + \bracket{e_2}{\courant{e_1}{e_3}}$,
\end{itemize}
where
$e, e_1, e_2, e_3 \in \Gamma(E)$, $f \in C^{\infty}(M)$ and
${\cal D}$ is a map from $C^{\infty}(M)$ to $\Gamma (E)$,
defined as
$\bracket{{\cal D}f}{e} = \rhoe(e) f$
\cite{LWX}.
\end{Definition}
The bilinear bracket $\courant{-}{-}$ is called the Dorfman bracket.
A Courant algebroid is encoded in the quadruple $(E, \bracket{-}{-}, \courant{-}{-}, \rhoe)$.

\begin{Example}\label{standardCA}
The \emph{standard Courant algebroid} is the Courant algebroid
as defined below on the vector bundle $E = TM \oplus T^*M$.

The three operations of the standard Courant algebroid
are defined as follows:
\begin{gather*}
	\bracket{u + \alpha}{v + \beta} = \iota_u \beta + \iota_v \alpha,\qquad
\rhott(u+\alpha) = u, \\
\courant{u + \alpha}{v + \beta} = [u, v] + \calL_u \beta - \iota_v \rd \alpha + \iota_u \iota_v H
\end{gather*}
for $u + \alpha, v + \beta \in \Gamma(TM \oplus T^*M)$,
where $u$, $v$ are vector fields, $\alpha$, $\beta$ are $1$-forms,
and $H \in \Omega^3(M)$ is a closed $3$-form on $M$.
\end{Example}

\begin{Definition}
A \emph{Dirac structure} $L$ is a
maximally isotropic subbundle of a Courant algebroid~$E$,
whose sections are closed under the Dorfman bracket,
i.e., $L$ is a subbundle of a~Courant algebroid satisfying
$
\bracket{e_1}{e_2}=0$ (isotropic),
$
 \courant{e_1}{e_2} \in \Gamma(L)$ (closed)
for every $e_1, e_2 \in \Gamma(L)$.
\end{Definition}

The following proposition is a basic fact for a Dirac structure.
\begin{Proposition}[\cite{LWX}]
A Dirac structure $L$ is a Lie algebroid.
\end{Proposition}

\begin{Example}\label{symplecticdirac}
Let $(M, \omega)$ be a pre-symplectic manifold and
\[(TM \oplus T^*M, \bracket{-}{-}, \courant{-}{-}, \allowbreak\rhott)\]
 be a
standard Courant algebroid with $H=0$ in Example \ref{standardCA}.
A bundle map $\omega^{\flat}\colon TM \rightarrow T^*M$ is defined
by \smash{$\omega^{\flat}(u)(v) = \omega(u, v)$} for every $v \in \mathfrak{X}(M)$.
A subbundle $L_{\omega} \subset TM \oplus T^*M$ given~by
\[
L_{\omega} = \mathrm{Gr}(\omega) := \bigl\{ u + \omega^{\flat}(u) \mid u \in \mathfrak{X}(M) \bigr\}\]
is a Dirac structure.
In fact,
\smash{$\bigl\langle u + \omega^{\flat}(u),v + \omega^{\flat}(v)\bigr\rangle=0$}
since $\omega(u, v) = - \omega(v, u)$,
and the concrete calculation gives
\smash{$
\bigl[ u + \omega^{\flat}(u),v + \omega^{\flat}(v)\bigr]
= [u, v] + \omega^{\flat}([u, v])$} if $\rd \omega =0$, which means that~$L_{\omega}$ is involutive.
\end{Example}

\begin{Example}\label{Poissondirac}
Let $(M, \pi)$ be a Poisson manifold and
$(TM \oplus T^*M, \bracket{-}{-}, \courant{-}{-}, \rhott)$ be a
standard Courant algebroid with $H=0$.
For a bundle map $\pi^{\sharp}\colon T^*M \rightarrow TM$,
a subbundle~${L_{\pi} \subset TM \oplus T^*M}$ given by
\[
L_{\pi} = \mathrm{Gr}(\pi) := \bigl\{ -\pi^{\sharp}(\alpha) + \alpha \mid \alpha \in \Omega^1(M) \bigr\}
\]
is a Dirac structure.
In fact,
$\bigl\langle-\pi^{\sharp}(\alpha) + \alpha,-\pi^{\sharp}(\beta) + \beta\bigr\rangle=0$
since $\pi(\alpha, \beta) = - \pi(\beta, \alpha)$,
and the concrete calculation gives
\smash{$
\bigl[-\pi^{\sharp}(\alpha) + \alpha,-\pi^{\sharp}(\beta) + \beta\bigr]
= -\pi^{\sharp}([\alpha, \beta]_{\pi}) + [\alpha, \beta]_{\pi}$}
if $\pi$ is a Poisson bivector field, which means that $L_{\pi}$ is involutive.
\end{Example}

\section{Basic curvatures and momentum sections}\label{sec3}
Let $(A, \pi, \nabla,\mu)$ be a Hamiltonian Lie algebroid over a Poisson manifold.
The anchor map $\rhoa$ is a Lie algebroid morphism $\rhoa\colon A \rightarrow TM$,
i.e., $\rhoa$ satisfies
\begin{gather}
 [\rhoa(e_1), \rhoa(e_2)] = \rhoa([e_1, e_2])
\label{identity001}
\end{gather}
for $e_1, e_2 \in \Gamma(A)$.
The covariant forms of equation~\eqref{identity001}
is
\begin{gather}
(\nabla_{\rhoa(e_1)} \rhoa)(e_2) - (\nabla_{\rhoa(e_2)} \rhoa)(e_1)
+ \bigl\langle\rhoa,{}^A T(e_1, e_2)\bigr\rangle= 0.
\label{identity111}
\end{gather}
In fact, equation~\eqref{identity111} is proved as follows.
Using the Leibniz rule of the connection and equation~\eqref{Etorsion}, we obtain
\begin{align*}
\bigl\langle\rhoa,{}^A T(e_1, e_2)\bigr\rangle ={}&
\rhoa (\nabla_{\rhoa(e_1)} e_2) - \rhoa (\nabla_{\rhoa(e_2)} e_1)
- \rhoa([e_1, e_2])
 \\
={}&
\nabla_{\rhoa(e_1)} (\rhoa(e_2)) - \nabla_{\rhoa(e_2)} (\rhoa (e_1))
- (\nabla_{\rhoa(e_!)} \rhoa)(e_2) + (\nabla_{\rhoa(e_2)} \rhoa) ( e_1)
 \\ &
- [\rhoa(e_1), \rhoa(e_2)]
 \\
={}&
\calL_{\rhoa(e_1)} (\rhoa(e_2)) - \calL_{\rhoa(e_2)} (\rhoa (e_1))
- [\rhoa(e_1), \rhoa(e_2)]
 \\ &
- (\nabla_{\rhoa(e_1)} \rhoa)(e_2) + (\nabla_{\rhoa(e_2)} \rhoa) ( e_1)
 \\
={}&
- (\nabla_{\rhoa(e_1)} \rhoa)(e_2) + (\nabla_{\rhoa(e_2)} \rhoa) ( e_1).
\end{align*}
The Jacobi identity of the Lie bracket $[-,-]$ is equivalent to the equation
\[
\nabla_{\rhoa(e_1)} {}^A T(e_2, e_3) - {}^A T\bigl(e_1, {}^A T(e_2, e_3)\bigr)
- R(\rhoa(e_1), \rhoa(e_2), e_3) + \operatorname{Cycl}(e_1, e_2, e_3)
= 0.
\]

Before we analyze a comomentum section and a Poisson map,
we consider a formula between a momentum section and the basic curvature.
From the definition of the Hamiltonian Lie algebroid over a Poisson manifold,
we obtain the following formula, which is the key identity in our paper.
\begin{Lemma}\label{basicms}
In a Hamiltonian Lie algebroid over a Poisson manifold,
$\pi^{\sharp} \bigl\langle{}^A S,\mu\bigr\rangle =0$,
where $\bracket{-}{-} =0$ is the pairing of $A$ and $A^*$.
\end{Lemma}
\begin{proof}
They are proved by using local coordinate expressions.

For a Lie algebroid $(A, \rho, [-,-])$ over a smooth manifold $M$,
let $x^i$ be a local coordinate on~$M$, $e_a \in \Gamma(A)$ be a basis of the fiber of $A$ and $e^a \in \Gamma(A^*)$ be the dual basis on $A^*$.
The basis of~$TM$ is denoted by $\partial_i = \tfrac{\partial}{\partial x^i}$.
$i$, $j$, etc. are indices on $M$ and $a$, $b$, etc. are indices on
the fiber of~$A$ and~$A^*$.
Upper indices are ones for $TM$ and $A$, and lower indices are ones
for $T^*M$ and~$A^*$.\footnote{Einstein summation convention is used for
sums in local coordinate expressions.}

Local coordinate expressions of the anchor map and the Lie bracket are
$\rho(e_a)= \rho^i_a(x) \partial_i$ and~${[e_a, e_b ] = C_{ab}^c(x) e_c}$, where $\rho^i_a(x)$ and $C_{ab}^c(x)$
are local functions.
Identities of the Lie algebroid are\looseness=-1
\begin{gather}
 \rho_a^j \partial_j \rho_{b}^i - \rho_b^j \partial_j \rho_{a}^i = C_{ab}^c \rho_c^i,
\qquad
 C_{ad}^e C_{bc}^d + \rho_a^i \partial_i C_{bc}^e + \operatorname{Cycl}(abc) = 0.\label{LAidentity1}
\end{gather}
The condition of Poisson bivector field $\pi$, $[\pi, \pi]_S=0$, is
\begin{gather}
 \pi^{il} \partial_l \pi^{jk} + \operatorname{Cycl}(ijk) = 0.
\label{Poissonidentity}
\end{gather}
A connection $1$-form $\omega_a^b = \omega_{ai}^b \rd x^i$
for the connection $\nabla$ is introduced as
$\nabla_i e_a := \omega_{ai}^b e_b$ for the basis, $e_a$.
Local coordinate expressions of an $A$-torsion,
a curvature and a basic curvature are given by
\begin{gather}
T_{ab}^c =
- C_{ab}^c + \rho_a^i \omega_{bi}^c - \rho_b^i \omega_{ai}^c,
\label{lcAtosion}
\\
R_{ijb}^a =
\partial_i \omega_{aj}^b - \partial_j \omega_{ai}^b
+ \omega_{ai}^c \omega_{cj}^b - \omega_{aj}^c \omega_{ci}^b,
\nonumber
\\
S_{iab}^{c} =
\nabla_i T_{ab}^c + \rho_b^j R_{ija}^c - \rho_a^j R_{ijb}^c,
 \nonumber \\
\phantom{S_{iab}^{c} }{}= - \partial_i C^c_{ab} - \omega_{di}^c C_{ab}^d
+ \omega_{ai}^d C_{db}^c + \omega_{bi}^d C_{ad}^c
+ \rho_a^j \partial_j \omega_{bi}^c
- \rho_b^j \partial_j \omega_{ai}^c
\nonumber \\
\phantom{S_{iab}^{c} =}{}+ \partial_i \rho_a^j \omega_{bj}^c
- \partial_i \rho_b^j \omega_{aj}^c
+ \omega_{ai}^d \rho_d^j \omega_{bj} ^c
- \omega_{bi}^d \rho_d^j \omega_{aj} ^c,
\label{lcbasiccurv}
\end{gather}
where the covariant derivative $\nabla_i T_{ab}^c$ is given by
\[
\nabla_i T_{ab}^c =
\partial_i T_{ab}^c
+ \omega_{di}^c T_{ab}^d - \omega_{ai}^d T_{db}^c - \omega_{bi}^d T_{ad}^c.
\]

Equations~(P1), (P2) and (P3) in the definition of a Hamiltonian Lie algebroid are
\begin{gather}
{\rm (P1)} \quad {}^A \nabla_a \pi^{ij}
= \rho_a^k \partial_k \pi^{ij} - \partial_k \rho^i_a \pi^{kj}
+ \rho^i_b \omega^b_{ak} \pi^{kj}
- \partial_k \rho^j_a \pi^{ik}
+ \rho^j_b \omega^b_{ak} \pi^{ik} = 0,\nonumber\\
{\rm (P2)} \quad
\rho^i_a = \pi^{ij} \nabla_j \mu_a
\bigl(= \pi^{ij} \bigl(\partial_j \mu_a - \omega_{aj}^b \mu_b\bigr)\bigr),
\label{Pms12}\\
{\rm (P3)}\quad
\rho^i_{[a} \partial_i \mu_{b]} - C_{ab}^c \mu_c
= - \pi^{ij} \nabla_i \mu_a \nabla_j \mu_b
\bigl(= \pi^{ij} \bigl(\partial_i \mu_a - \omega_{ai}^c \mu_c\bigr)
\bigl(\partial_j \mu_b - \omega_{bj}^d \mu_d\bigr)\bigr).
\label{Pms13}
\end{gather}

Substituting \eqref{Pms12} into the identity \eqref{LAidentity1},
$\rho_a^j \partial_j \rho_{b}^i - \rho_b^j \partial_j \rho_{a}^i = C_{ab}^c \rho_c^i$,
we obtain the following equation:
\begin{gather}
 \pi^{ij} \bigl(-\pi^{kl} \partial_k \nabla_j \mu_a \nabla_l \mu_b
- \pi^{kl} \nabla_k \mu_a \partial_l \nabla_j \mu_b
- \partial_j \pi^{kl} \nabla_k \mu_a \nabla_l \mu_b
- C_{ab}^c \nabla_j \mu_c\bigr)=0.
\label{LAIden111}
\end{gather}
Here, we use $\pi$ is the Poisson bivector field, i.e.,
equation~\eqref{Poissonidentity}.
Moreover, equation~\eqref{LAIden111} is equivalent to
\begin{gather}
\pi^{ij} \bigl(\pi^{kl} \nabla_k \nabla_j \mu_a \nabla_l \mu_b
+ \pi^{kl} \nabla_k \mu_a \nabla_l \nabla_j \mu_b
+ \partial_j \pi^{kl} \nabla_k \mu_a \nabla_l \mu_b
- T_{ab}^c \nabla_j \mu_c\bigr)=0,
\label{identity11}
\end{gather}
which is the covariant expression of equation~\eqref{LAIden111}.

We consider another identity as follows.
Substituting \eqref{Pms12} to \eqref{Pms13}
and using equation~\eqref{Poissonidentity},
we obtain
\[
 \pi^{ij} \nabla_j \mu_a \partial_i \mu_b
- \pi^{ij} \nabla_j \mu_b \partial_i \mu_a
+ \pi^{ij} \nabla_i \mu_a \nabla_j \mu_b
- C_{ab}^c \mu_c=0.
\]
The covariant expression of this identity is
\begin{gather}
\pi^{kl} \nabla_k \mu_a \nabla_l \mu_b
- T_{ab}^c \mu_c=0,
\label{P3covlocal}
\end{gather}
where equation~\eqref{lcAtosion} is used.
The global form of equation~\eqref{P3covlocal} is
\[
\pi(\nabla \mu(e_1), \nabla \mu(e_2))
- \big\langle\mu,{}^A T(e_1, e_2)\big\rangle = 0
\]
for $e_1, e_2 \in \Gamma(A)$.

Acting the covariant derivative by $\nabla_j$ to equation~\eqref{P3covlocal},
and using the identity \eqref{lcbasiccurv},
\smash{$S_{iab}^{c} = \nabla_i T_{ab}^c + \rho_b^j R_{ija}^c - \rho_a^j R_{ijb}^c$},
we obtain the equation
\begin{gather}
 \pi^{kl} \nabla_k \nabla_j \mu_a \nabla_l \mu_b
+ \pi^{kl} \nabla_k \mu_a \nabla_l \nabla_j \mu_b
+ \nabla_j \pi^{kl} \nabla_k \mu_a \nabla_l \mu_b
- T_{ab}^c \nabla_j \mu_c
- S_{jab}^c \mu_c =0.
\label{identity12}
\end{gather}

By taking $\pi^{ij} \times$ \eqref{identity12},
we obtain the identity
\begin{gather}
 \pi^{ij}\bigl(\pi^{kl} \nabla_k \nabla_j \mu_a \nabla_l \mu_b
+ \pi^{kl} \nabla_k \mu_a \nabla_l \nabla_j \mu_b
+ \nabla_j \pi^{kl} \nabla_k \mu_a \nabla_l \mu_b\nonumber\\
\qquad
{}- T_{ab}^c \nabla_j \mu_c
- S_{jab}^c \mu_c\bigr) =0.
\label{identity13}
\end{gather}
Compare equations~\eqref{identity11} and~\eqref{identity13},
where note that $\nabla_j \pi^{kl} = \partial_j \pi^{kl}$
since we do not introduce a $TM$-connection on $TM$.\footnote{We have introduced a $TM$-connection on $A$, $\nabla$ and
an $A$-connection on $TM$, \smash{$\anabla = \nablabas$}.}
Consistency between equations~\eqref{identity11} and~\eqref{identity13}
gives the following identity,
$\pi^{ij} S_{jab}^c \mu_c =0$,
i.e.,
$ \pi^{\sharp} \bracket{{}^AS}{\mu} =0$.
\end{proof}

Lemma \ref{basicms} is equivalent to the following statement.
\begin{Lemma}\label{basicms2}
In a Hamiltonian Lie algebroid over a Poisson manifold,
$\bigl\langle{}^A S,\mu\bigr\rangle \in \ker\bigl(\pi^{\sharp}\bigr)$
for a~momentum section $\mu$.
\end{Lemma}
Now, we assume the equation
\begin{gather}
\bigl\langle{}^A S,\mu\bigr\rangle(e_1, e_2)
= 0
\label{basiccurmu}
\end{gather}
for every $e_1, e_2 \in \Gamma(A)$.
In the local coordinate, equation~\eqref{basiccurmu} is
\begin{gather}
S_{jab}^c \mu_c =0.
\label{basiccurmu2}
\end{gather}
Note that the left-hand side of equation~\eqref{basiccurmu} is
the commutator of the connection and the $A$-connection on $\mu$,
\[
\bigl\langle\iota_v {}^A S(e_1, e_2),\mu\bigr\rangle =
\frac{1}{2} \bigl(- \bigl(\bigl[\nabla_{v}, \anabla_{e_1}\bigr] \mu\bigr)(e_2)
+ \bigl(\bigl[\nabla_{v}, \anabla_{e_2}\bigr] \mu\bigr)(e_1)\bigr).
\]
Therefore, equation~\eqref{basiccurmu} is regarded as
$\nabla_{v}$ and $\anabla$ are commutative on the momentum section~$\mu$.

The equation \eqref{basiccurmu} is satisfied if $\ker\bigl(\pi^{\sharp}\bigr)= 0$, or
sections $e_1, e_2 \in \Gamma(A)$ are restricted to the subspace of $\Gamma(A)$
satisfying equation~\eqref{basiccurmu}.
If $\pi$ is nondegenerate, $\omega = \pi^{-1}$ is symplectic and
$\ker\bigl(\pi^{\sharp}\bigr)= 0$ is always satisfied.
However, we do not specify geometry of equation~\eqref{basiccurmu} to
symplectic in this paper.

Substituting equation~\eqref{basiccurmu2} into equation~\eqref{identity13},
the following identity is obtained:
\begin{gather}
 \pi^{kl} \nabla_k \nabla_j \mu_a \nabla_l \mu_b
+ \pi^{kl} \nabla_k \mu_a \nabla_l \nabla_j \mu_b
+ \nabla_j \pi^{kl} \nabla_k \mu_a \nabla_l \mu_b
- T_{ab}^c \nabla_j \mu_c =0.
\label{identity22}
\end{gather}
We consider a coordinate independent equation of equation~\eqref{identity22}.
For it, we consider
the Koszul bracket $[-,-]_{\pi}\colon\Omega^1(M) \times \Omega^1(M)
\rightarrow \Omega^1(M)$,
\begin{gather}
[\alpha, \beta]_{\pi} =
\calL_{\pi^{\sharp} \alpha} \beta - \calL_{\pi^{\sharp} \beta} \alpha
- \rd (\pi(\alpha, \beta)).
\label{Koszulbracket}
\end{gather}
The local coordinate expression of this bracket $[-, -]_{\pi}$ is
\[
([\alpha, \beta]_{\pi})_j
=
\pi^{kl} \partial_k \alpha_j \beta_l
+ \pi^{kl} \alpha_k \partial_l \beta_j
+ \partial_j \pi^{kl} \alpha_k \beta_l
\]
for $1$-form $\alpha = \alpha_i(x) \rd x^i$ and $\beta = \beta_i(x) \rd x^i$.
Using this bracket, equation~\eqref{identity22} is
the following coordinate independent equation:
\[
\iota_v ([(\nabla \mu)(e_1), (\nabla \mu)(e_2)]_{\pi}{})
+ (\nabla_{v} \mu)([e_1, e_2])
=0
\]
or
\begin{gather}
- [(\nabla \mu)(e_1), (\nabla \mu)(e_2)]_{\pi}
= (\nabla \mu)([e_1, e_2]).
\label{identity34}
\end{gather}
Therefore, if we define the map $(\nabla \mu)^*\colon \Gamma(A) \rightarrow \Gamma(T^*M)$ induced from $\nabla \mu$ as
$(\nabla \mu)^*(e) := \bracket{\nabla \mu}{e}$ for every $e \in \Gamma(A)$,
$-(\nabla \mu)^*$ is a Lie algebra morphism from $\Gamma(A)$ to $\Gamma(T^*M)$.
Similar to $\mu^*$,
since $\nabla \mu$ is a $1$-form taking value on $A^*$,
$\nabla \mu \in \Omega^1(M, A^*)$ and $(\nabla \mu)^*$ is defined from
the pairing of $A$ and $A^*$, $(\nabla \mu)^*$ maps a vector
$a(x) \in A_x$ on the fiber of a point $x \in M$ to a covector
$\alpha(x) \in T^*_x M$. Thus $(\nabla \mu)^*$ is regarded as
the bundle map, $(\nabla \mu)^*\colon A \rightarrow T^*M$.

The condition (P2), equation~\eqref{PMS2},
$- \pi^{\sharp} (\nabla \mu)(e) = \rhoa(e)$
is nothing but the condition for anchor maps in the Lie algebroid morphism
$-(\nabla \mu)^*\colon A \rightarrow TM$,
where the anchor map in $A$ is~$\rhoa$ and one in $T^*M$ is $\pi^{\sharp}$.
Therefore, we have proved that $(\nabla \mu)^*$ is a Lie algebroid morphism.

We summarize analysis in this subsection.
If $M$ is a symplectic manifold, $\bigl\langle{}^A S,\mu\bigr\rangle =0$ is satisfied and we obtain the following proposition.\footnote{For a momentum section $\mu \in \Gamma(A^*)$, notation $\mu(e)$ in fact means that the pairing of $A^*$ and $A$, $\bracket{\mu}{e}$. After this section, if $\mu$ is regarded as the map $\mu^*\colon\Gamma(A) \rightarrow C^{\infty}(M)$ to consider a morphism, it is denoted by $\mu^*$ to emphasize it, which is a \textit{comomentum} defined by $\mu^*(e) = \bracket{\mu}{e}$.}
\begin{Proposition}\label{nablamuLAmorphismS}
Let $(M, \pi, A, \mu)$ be a Hamiltonian Lie algebroid over a symplectic manifold.
Then $-(\nabla \mu)^*\colon A \rightarrow T^*M$ is a Lie algebroid morphism.
\end{Proposition}
For a Poisson manifold $M$, the corresponding proposition is obtained.
\begin{Proposition}\label{nablamuLAmorphism}
Let $(M, \pi, A, \mu)$ be a Hamiltonian Lie algebroid over a Poisson manifold.
Then if $\bigl\langle{}^A S,\mu\bigr\rangle =0$, $-(\nabla \mu)^*\colon A \rightarrow T^*M$ is a Lie algebroid morphism.
\end{Proposition}

\section{Comomentum sections and Lie algebroid morphisms}\label{sec:CMSLAM}
We analyze the comomentum description of momentum sections and Hamiltonian Lie algebroids. The comomentum section gives a Lie algebroid morphism from $A$ to a proper vector bundle on~$M$.

\subsection{Pre-symplectic case}\label{presymcom}
Let $(M, \omega)$ be a pre-symplectic manifold.
We define the following bracket on
$\mathfrak{X}(M) \oplus C^{\infty}(M)$,
\begin{gather}
\liebra{u + f}{v + g}
= [u, v] + u g - v f - \iota_u \iota_v \omega,
\label{Lieline}
\end{gather}
where $u, v \in \mathfrak{X}(M)$ and $f, g \in C^{\infty}(M)$.
The bracket \eqref{Lieline} is the $\bR$-bilinear Lie bracket
if and only if $\rd \omega = 0$.
Moreover, we define the map,
$\rhotr\colon \mathfrak{X}(M) \oplus C^{\infty}(M)
\rightarrow \mathfrak{X}(M)$,
\begin{gather}
\rhotr(u + f) = u.
\label{Lieanchor}
\end{gather}
The space for $\mathfrak{X}(M) \oplus C^{\infty}(M)$ is denoted by
$TM \oplus \bR$.
An element $u + f \in \mathfrak{X}(M) \oplus C^{\infty}(M)$ is regarded as a
section on the bundle $TM \oplus \bR$.
Then the map \eqref{Lieanchor} is induced from the bundle map
$\rhotr\colon TM \oplus \bR \rightarrow TM$.
Two operations \eqref{Lieline} and \eqref{Lieanchor}
define a Lie algebroid on~${TM \oplus \bR}$.
The Lie algebroid $(TM \oplus \bR, \liebra{-}{-}, \rhotr)$ is
 called the Almeida--Molino Lie algebroid~\cite{Almedia}.

Let $(A, \rhoa, [-,-]_A)$ be a Lie algebroid over $M$.
For a section $\mu \in \Gamma(A^*)$,
a map $\mu^*\colon\Gamma(A) \rightarrow C^{\infty}(M)$
is defined by $\bracket{\mu}{e} = \mu^*(e)$.
Since $\mu$ is a section on $A^*$ and $\mu^*$ is defined from
the pairing of $A$ and $A^*$, $\mu^*$ maps a vector
$a(x) \in A_x$ on the fiber of a point $x \in M$ to a value of a~function
in $x \in M$, $f(x) \in \bR$.
Thus, $\mu^*$ is regarded as
the bundle map $\mu^*\colon A \rightarrow M \times \bR$ that
is the identity map on $M$.
Then $\rhoa + \mu^*$ is a map
$\rhoa + \mu^*\colon \Gamma(A) \rightarrow \mathfrak{X}(M) \oplus C^{\infty}(M)$,
which is induced from the bundle map
$\rhoa + \mu^*\colon A \rightarrow TM \oplus \bR$.
We obtain the following proposition to characterize the condition of a momentum section (S3) as a Lie algebroid morphism.
\begin{Proposition}\label{SympS3}
$\rhoa + \mu^*$ is a Lie algebroid morphism from $A$ to $TM \oplus \bR$
if and only if $\mu$
satisfies the condition $\rm (S3)$, i.e., equation~\eqref{HH3}.
\end{Proposition}

\begin{proof}
We prove the equation
\begin{gather}\label{SympS301}
\liebra{(\rhoa + \mu^*)(e_1)}{(\rhoa + \mu^*)(e_2)} =
(\rhoa + \mu^*)([e_1, e_2]_A)
\end{gather}
for every $e_1, e_2 \in \Gamma(A)$.
The identity
$[\rhoa(e_1), \rhoa(e_2)]_{TM} = \rhoa([e_1, e_2]_A)$,
is satisfied
since $\rhoa$ is a Lie algebroid morphism from $A$ to $TM$
from the definition of a Lie algebroid.
It is nothing but the $\rhoa([e_1, e_2]_A)$ part in equation~\eqref{SympS301}.
Since the Lie algebroid differential ${}^A \rd$ is
\[
{}^A \rd \mu (e_1, e_2)
=
\calL_{\rhoa(e_1)} (\mu(e_2)) - \calL_{\rhoa(e_2)} (\mu(e_1))
- \mu([e_1, e_2]_A),
\]
the condition (S3) is equivalent to
\begin{gather}
\calL_{\rhoa(e_1)} (\mu(e_2)) - \calL_{\rhoa(e_2)} (\mu(e_1))
+ \iota_{\rhoa(e_2)} \iota_{\rhoa(e_1)} \omega
= \mu([e_1, e_2]_A),
\label{LAmorphism01}
\end{gather}
which is the $\mu^*([e_1, e_2])$ part in equation~\eqref{SympS301}.
Therefore, \eqref{SympS301} is proved, which is the condition~\eqref{LAmorphism01} for a Lie bracket
in the definition of a Lie algebroid morphism.
The second condition~\eqref{LAmorphism02} for the anchor map,
\begin{gather}
\rhoa (e) =\rhotr \circ (\rhoa + \mu^*)(e),
\label{LAmorphism02+}
\end{gather}
is easily checked.
Equations~\eqref{SympS301} and \eqref{LAmorphism02+} show that $\rhoa+\mu^*$ is a Lie algebroid morphism from~$A$ to $TM \oplus \bR$.
\end{proof}

Next, we characterize the condition (S2) in terms of a Lie algebroid morphism.

For a section $\mu \in \Gamma(A^*)$, $\nabla \mu \in \Omega^1(M, A^*)$
induces a map $(\nabla_X \mu)^*\colon \Gamma(A) \rightarrow \Omega^1(M)$
by defining $\bracket{\nabla_X \mu}{e} = (\nabla_X \mu)^*(e)$
for every $e \in \Gamma(A)$ and $X \in \mathfrak{X}(M)$.
Since $(\nabla_X \mu)^*$ is defined by the pairing between $A$ and $A^*$
and $(\nabla_X \mu)^*|_M = \mathrm{id}$,
$(\nabla_X \mu)^*$ is considered as a bundle map~${(\nabla \mu)^*\colon A \rightarrow T^*M}$.
We consider the map
$\rhoa - (\nabla \mu)^*\colon A \rightarrow TM \oplus T^*M$.

$TM \oplus T^*M$ naturally has a standard Courant algebroid structure
with $H=0$ in Example~\ref{standardCA}.
The pre-symplectic structure on $M$ induces the Dirac structure $L_{\omega}$.
Then the following proposition is obtained.
\begin{Proposition}\label{SympS2}
The map $\rhoa - (\nabla \mu)^*\colon A \rightarrow TM \oplus T^*M$
is a Lie algebroid morphism ${\rhoa - (\nabla \mu)^*\colon A \rightarrow L_{\omega}}$
if and only if $\mu$
satisfies the condition $\rm (S2)$, i.e., equation~\eqref{HH2}.
\end{Proposition}

\begin{proof}
Assume the condition (S2). For $e_1, e_2 \in \Gamma(A)$,
using equation~\eqref{HH2}, (S2), we obtain the inner product
\[
\bracket{(\rhoa - (\nabla \mu)^*)(e_1)}{(\rhoa - (\nabla \mu)^*)(e_2)}
= - \omega(\rho(e_1), \rho(e_2)) - \omega(\rho(e_2), \rho(e_1)) =0,
\]
and the Dorfman bracket
\begin{gather*}
 \courant{(\rhoa - (\nabla \mu)^*)(e_1)}{(\rhoa - (\nabla \mu)^*)(e_2)}\\
\qquad= [\rhoa(e_1), \rhoa(e_2)]_{TM}
+ \iota_{[\rhoa(e_1), \rhoa(e_2)]_{TM}}\omega
- \iota_{\rhoa(e_2)} \iota_{\rhoa(e_1)} \rd \omega
 \\
\qquad= \rhoa([e_1, e_2]_A)
+ \iota_{\rhoa([e_1, e_2]_A)}\omega
 \\
\qquad= \rhoa([e_1, e_2]_A) - (\nabla \mu)^*([e_1, e_2]_A),
\end{gather*}
which show $(\rhoa - (\nabla \mu)^*)(e) \in L_{\omega}$
since $\rd \omega =0$.
The condition \eqref{LAmorphism02} for anchor maps are obvious.
Therefore, $\rho - (\nabla \mu)^*$ is a Lie algebroid morphism.
If we consider the reverse, the condition (S2) is obtained from the Lie algebroid morphism condition.
\end{proof}

By combining two Propositions \ref{SympS3} and \ref{SympS2}, we obtain a characterization of the momentum section based on Lie algebroid morphisms.
\begin{Theorem}
Let $(M, \omega)$ be a pre-symplectic manifold and $(A, \rhoa, [-,-]_A)$ be a Lie algebroid over $M$.
$\mu \in \Gamma(A^*)$ is a bracket-compatible momentum section
if and only if $\rhoa + \mu^*$ is a Lie algebroid morphism from
$A$ to $TM \oplus \bR$ and
$\rhoa - (\nabla \mu)^*$ is a Lie algebroid morphism
from $A$ to~$L_{\omega}$.
\end{Theorem}
The map $\rhoa - (\nabla \mu)^*$ can be considered as a Lie algebroid morphism from $A$ to $TM$ since $L_{\omega} \simeq TM$.

We define a bracket-compatible comomentum section.
\begin{Definition}[comomentum sections over pre-symplectic manifolds]\label{comomSymp}
Let $(M, \omega)$ be a pre-symplectic manifold
and $(A, [-,-], \rhoa)$ be a Lie algebroid over $M$.
$\mu^*\colon A \rightarrow M \times \bR$ is called a~bracket-compatible
if $\rhoa + \mu^*$ is a Lie algebroid morphism from
$A$ to $TM \oplus \bR$.
$\mu^*$ is called a~bracket-compatible comomentum section if in addition,
$\rhoa - (\nabla \mu)^*$ is a Lie algebroid morphism
from $A$ to $L_{\omega}$.
\end{Definition}

\subsection{Poisson case}\label{PoisLAM}

Let $(M, \pi)$ be a Poisson manifold.
We consider the following bilinear bracket on
$\Omega^1(M) \oplus C^{\infty}(M)$:
\begin{gather}
[\alpha + f, \beta + g]_{T^*M \oplus \bR} =
[\alpha, \beta]_{\pi} - \pi^{\sharp}(\alpha) g
+ \pi^{\sharp}(\beta) f - \pi(\alpha, \beta).
\label{PoissonTMR}
\end{gather}
where $\alpha, \beta \in \Omega^1(M)$ and $f, g \in C^{\infty}(M)$.
$[-,-]_{\pi}$ is the Koszul bracket defined in equation~\eqref{Koszulbracket}.
The bracket \eqref{PoissonTMR} is a Lie bracket if and only if $\pi$ is a Poisson bivector field.
If $\alpha = \rd h$ and~${\beta = \rd k}$ with $h, k \in C^{\infty}(M)$,
the equation \eqref{PoissonTMR} reduces to
\[
[\rd h + f, \rd k + g]_{T^*M \oplus \bR} =
\rd (\sbv{h}{k}_{M}) - \sbv{f}{g}_{M},
\]
where $\sbv{f}{g}_{M} = \pi(\rd f, \rd g)$ is a Poisson bracket on $C^{\infty}(M)$ defined by $\pi$.

We define a bundle map,
$\rhotr\colon T^*M \oplus \bR \rightarrow TM$ as
\begin{gather}
\rhotr(\alpha + f) = \pi^{\sharp} (\alpha).
\label{Lieanchor02}
\end{gather}
Operations \eqref{PoissonTMR} and \eqref{Lieanchor02} give a Lie algebroid
structure on $T^*M \oplus \bR$.
The Lie algebroid $(T^*M \oplus \bR, \rhotr, [-,-]_{T^*M \oplus \bR})$
is a Poisson analogue of the Almeida--Molino Lie algebroid in Section
\ref{presymcom}.

Similar to Section \ref{presymcom}, we consider a map
$\mu^*\colon\Gamma(A) \rightarrow C^{\infty}(M)$ on
a Lie algebroid $A$ over~$M$.
Then $(\nabla \mu)^*$ is a map
$(\nabla \mu)^*\colon \Gamma(A) \rightarrow \Omega^1(M)$
and $\bigl({}^A \rd \mu\bigr)^*$ is a map
$\bigl({}^A \rd \mu\bigr)^*\colon \Gamma(A) \times \Gamma(A) \rightarrow C^{\infty}(M)$.

A \textit{comomentum section} on a Poisson manifold is defined from
equations~\eqref{PMS2} and \eqref{PMS3} as follows.
\begin{Definition}[comomentum sections over Poisson manifolds]\label{comomPois}
Let $(M, \pi)$ be a Poisson manifold with a Poisson bivector field $\pi$
and $(A, [-,-], \rhoa)$ be a Lie algebroid over $M$.

A map $\mu^*\colon \Gamma(A) \rightarrow C^{\infty}(M)$ is called a bracket-compatible $\nabla$-comomentum section if it satisfies
\begin{gather}
{\rm (PC2)} \quad
 \rhoa(e) = - \pi^{\sharp} ((\nabla \mu)^*(e)),
\label{PCMS2}
\\
{\rm (PC3)} \quad \bigl(\bigl({}^A \rd \mu\bigr)^*\bigr) (e_1, e_2) = - \pi((\nabla \mu)^*)(e_1), ((\nabla \mu)^*)(e_2))
\label{PCMS3}
\end{gather}
for $e, e_1, e_2 \in \Gamma(A)$.
\end{Definition}

We give characterizations similar to Propositions \ref{SympS2} and \ref{SympS3}
over a pre-symplectic manifold.

Consider the map
$-(\nabla \mu)^* + \mu^*\colon \Gamma(A) \rightarrow \Omega^1(M) \oplus C^{\infty}(M)$,
which can be regarded as the bundle map,
$- (\nabla \mu)^* + \mu^*\colon A \rightarrow T^*M \oplus \bR$.
Then we obtain the following proposition.
\begin{Proposition}\label{PoisP3}
Let $\mu$ be a momentum section satisfying equations~\eqref{PCMS2} and \eqref{PCMS3}. Equivalently, let $\mu^*$ be a comomentum section.
Assume $\bigl\langle {}^A S,\mu\bigr\rangle = 0$ in equation~\eqref{basiccurmu}.
Then $-(\nabla \mu)^* + \mu^*$ is a Lie algebroid morphism
from $A$ to $T^*M \oplus \bR$.
\end{Proposition}

\begin{proof}
We prove that the equation
\begin{gather}
[(-(\nabla \mu)^* + \mu^*)(e_1), (-(\nabla \mu)^* + \mu^*)(e_2)]_{T^*M \oplus \bR} =
(-(\nabla \mu)^* + \mu^*)([e_1, e_2]_A)
\label{PoisLAmorphsm01}
\end{gather}
is satisfied under the Lie bracket \eqref{PoissonTMR} in $T^*M \oplus \bR$
for every $e_1, e_2 \in \Gamma(A)$.

If $\bigl\langle {}^A S,\mu\bigr\rangle = 0$ is imposed, equation~\eqref{identity34} is satisfied.
The left-hand side of the $T^*M$ part of equation~\eqref{PoisLAmorphsm01} is
$
 [-(\nabla \mu)^*(e_1), -(\nabla \mu)^*(e_2)]_{\pi}
$.
Using equation~\eqref{identity34}, it is equal to
$-(\nabla \mu)^*([e_1, e_2])$, which is the $T^*M$ part in the right-hand side of equation~\eqref{PoisLAmorphsm01}.

Using \eqref{PMS2}, the $\bR$ part of equation~\eqref{PoisLAmorphsm01} is
\[
\calL_{\rhoa(e_1)} (\mu(e_2)) - \calL_{\rhoa(e_2)} (\mu(e_1))
+ \pi((\nabla \mu)^*(e_1), (\nabla \mu)^*(e_2)).
\]
It is equal to $\mu^*([e_1, e_2)]$ from (P2) and equation~\eqref{PMS3}.

The morphism for the anchor maps is
\[
\rhoa(e) = \rhotr \circ (-(\nabla \mu)^* + \mu^*)(e),
\]
which is equivalent to the condition (PC2).
Equations~\eqref{PoisLAmorphsm01} and \eqref{PoisLAmorphsm01} show that
$-(\nabla \mu)^* + \mu^*$ is the Lie algebroid morphism
$A \rightarrow T^*M \oplus \bR$.
\end{proof}

Next, a Lie algebroid morphism
corresponding to Proposition \ref{SympS2} is introduced.

We consider the map $\rhoa - (\nabla \mu)^*\colon A \rightarrow TM \oplus T^*M$.
Here $TM \oplus T^*M$ is the standard Courant algebroid
with $H=0$ in Example \ref{standardCA}.

\begin{Proposition}\label{PoisP2}
Let $\mu$ be a comomentum section satisfying $\rm (PC2)$ and $\rm (PC3)$.
Assume $\bigl\langle {}^A S,\mu\bigr\rangle = 0$.
Then $\rhoa - (\nabla \mu)^*$ is a Lie algebroid morphism
from $A$ to $L_{\pi}$.
\end{Proposition}

\begin{proof}
From the condition \eqref{PCMS2}, the inner product is
\begin{gather}
 \bracket{(\rhoa - (\nabla \mu)^*)(e_1)}{(\rhoa - (\nabla \mu)^*)(e_2)}
\nonumber \\
\qquad= - \pi((\nabla \mu)^*(e_1), (\nabla \mu)^*(e_2))
- \pi((\nabla \mu)^*(e_2), (\nabla \mu)^*(e_1)) =0.
\label{Poisinner}
\end{gather}
Using (PC2) and the $\pi$ is Poisson, the Dorfman bracket is
\begin{gather}
[(\rhoa - (\nabla \mu)^*)(e_1), (\rhoa - (\nabla \mu)^*)(e_2)]_D\nonumber\\
\qquad= \pi^{\sharp}[(\nabla \mu)^*(e_1), (\nabla \mu)^*(e_2)]_{\pi}
+ [(\nabla \mu)^*(e_1), (\nabla \mu)^*(e_2)]_{\pi},
\label{PoisLie21}
\end{gather}
Using equation~\eqref{identity34}, this equation~\eqref{PoisLie21}
becomes
\begin{gather}
= - \pi^{\sharp} (\nabla \mu)^*([e_1, e_2]_A)
- (\nabla \mu)^*([e_1, e_2]_A)
= (\rho - (\nabla \mu)^*)([e_1, e_2]_A),
\label{PoisLie22}
\end{gather}
which gives
\begin{gather}
[(\rhoa - (\nabla \mu)^*)(e_1), (\rhoa - (\nabla \mu)^*)(e_2)]_D
= (\rho - (\nabla \mu)^*)([e_1, e_2]_A).
\label{PoisLie01}
\end{gather}
Equations~\eqref{Poisinner} and \eqref{PoisLie22} show that
$(\rhoa - (\nabla \mu)^*)(e)$ is an element of the Dirac structure
$L_{\pi}$.
Equation~\eqref{PoisLie01} means that $(\rhoa - (\nabla \mu)^*)$ is
a Lie algebroid morphism between $A$ and $L_{\pi}$.

The condition \eqref{LAmorphism02} for anchor maps are obvious.
Therefore, $\rho - (\nabla \mu)^*$ is a Lie algebroid morphism
from $A$ to $L_{\pi}$.
\end{proof}

We combine two Propositions \ref{PoisP3} and \ref{PoisP2}.
\begin{Theorem}
Let $(M, \pi)$ be a Poisson manifold and $(A, \rhoa, [-,-]_A)$ be a Lie algebroid over~$M$.
Let $\mu^*\colon \Gamma(A) \rightarrow C^{\infty}(M)$ be a bracket-compatible comomentum section. If $\bigl\langle {}^A S,\mu\bigr\rangle = 0$,
$-(\nabla \mu)^* + \mu^*$ is a Lie algebroid morphism from
$A$ to $T^*M \oplus \bR$ and
$\rhoa - (\nabla \mu)^*$ is a Lie algebroid morphism
from $A$ to $L_{\pi}$.
\end{Theorem}
Under the condition $\bigl\langle {}^A S,\mu\bigr\rangle = 0$,
the reverse is proved, i.e.,
if $-(\nabla \mu)^* + \mu^*$ is a Lie algebroid morphism from
$A$ to $T^*M \oplus \bR$ and
$\rhoa - (\nabla \mu)^*$ is a Lie algebroid morphism
from $A$ to $L_{\pi}$,
$\mu^*\colon \Gamma(A) \rightarrow C^{\infty}(M)$ be a bracket-compatible comomentum section.

\section{Momentum sections as Poisson maps}\label{sec5}
In this section, we construct a Poisson map
by considering a map from $T^*M \oplus \bR$ to $A^*$ induced from a momentum section $\mu$.
It is an improvement of a ``bivector map'' in the paper
\cite{Blohmann:2023} to a~Poisson map in some sense.
In this section, we concentrate on a Hamiltonian Lie algebroid over a Poisson manifold. Obviously, a Hamiltonian Lie algebroid over a symplectic manifold is a~special case.

\begin{Definition}[Poisson map]
Let $(M_1, \pi_1)$ and $(M_2, \pi_2)$ be Poisson manifolds.
A smooth map~$\psi$ is a \textit{Poisson map} if
$\psi^*\colon C^{\infty}(M_2) \rightarrow C^{\infty}(M_1)$ satisfies
\[
\psi^* \sbv{f}{g}_2 = \sbv{\psi^* f}{\psi^* g}_1
\]
for $f, g \in C^{\infty}(M_2)$, where $\sbv{-}{-}_1$
and $\sbv{-}{-}_2$ are Poisson brackets induced from $\pi_1$ and $\pi_2$.
\end{Definition}
The condition of a Poisson map is equivalent to the condition for Poisson bivector fields,
$
\psi_* \pi_1 = \pi_2$.

As explained in introduction, a momentum map $\mu\colon M \rightarrow \mathfrak{g}^*$
is a Poisson map, where a~Poisson structure on $M$ is induced from the symplectic structure, or $M$ is a Poisson manifold. A Poisson structure on $\mathfrak{g}^*$ is the KKS Poisson structure.
We generalize it to the momentum section.

For construction of a Poisson map,
we can use the following fundamental facts.
For instance, one can refer to \cite[Theorem 13.70]{CrainicFernandesMurcut} or
 \cite[Proposition 2.13]{Meinrenken02}.
The first one is that
if $A$ is a Lie algebroid, $A^*$ is a Poisson manifold
with a (fiberwise linear) Poisson structure.
The second one is as follows.
\begin{Proposition}\label{dualPoisson}
Let $(A_1, M_1)$ and $(A_2, M_2)$ be two Lie algebroids.
A vector bundle morphism~${\phi\colon A_1 \rightarrow A_2}$ is Lie algebroid morphism
if and only if the dual map $\psi\colon A_2^* \rightarrow A_1^*$ is a~Poisson map.
\end{Proposition}

Applying Proposition \ref{dualPoisson} to the map
$-(\nabla \mu)^*$ in Proposition \ref{nablamuLAmorphism},
we obtain the following proposition.

\begin{Proposition}\label{nablamuPoisson}
Let $(M, \pi, A, \mu)$ be a Hamiltonian Lie algebroid over a Poisson manifold.
Then if $\bigl\langle{}^A S,\mu\bigr\rangle =0$,
$-\nabla \mu\colon TM \rightarrow A^*$ is a Poisson map.
\end{Proposition}
However, since the map is not based on $\mu$ but $\nabla \mu$,
this Poisson map is not what we want.

\subsection{Dg manifolds for Lie algebroids and graded Poisson manifolds}

Let $(M, \pi)$ be a Poisson manifold.
We generalize the Lie bracket \eqref{PoissonTMR} on
$\Gamma(T^*M \oplus \bR)$ introduced in Section \ref{PoisLAM}
to a Poisson bracket.
For it, we consider ``a space of functions on $T^*M \oplus \bR$''.
Similarly we also consider ``a space of functions on $A^*$''
and a Poisson bracket on the space.
In other words, one can see $T^*M \oplus \bR$ and $A^*$
as graded manifolds.
Differential graded manifolds (dg manifolds), which are also called Q-manifolds
are graded manifolds with a graded vector field~$Q$ such that $Q^2=0$.
Dg manifolds corresponding to these spaces are introduced and (graded) Poisson brackets are defined on spaces of functions on these graded manifolds.
For reviews of graded manifolds, One can refer to \cite{Cattaneo:2010re, Ikeda:2012pv,
Roytenberg:2006qz} and references therein.

At first, we consider a graded manifold for $T^*M \oplus \bR$.
We introduce the graded manifold~${\calM = T^*[2](T[1]M \oplus \bR[1])}$,
where $[1]$ means that degree of coordinates is shifted by one.
Take local coordinates on $T[1]M \oplus \bR[1]$,
$\bigl(x^i, \eta^i, s\bigr)$ of degree $(0, 1, 1)$, where
$x^i$ is a local coordinate on $M$,
$\eta^i$ is the fiber coordinate on $T[1]M$,
and $s$ is one on $\bR[1]$.
Moreover, take fiber coordinates on
$T^*[2](T[1]M \oplus \bR[1])$, $(\xi_i, y_i, t)$ of degree
$(2, 1, 1)$
corresponds to a $1$-form $a_i(x) \rd x^i \in \Omega^1(M)$.
Here $C_k^{\infty}(\calM)$ is a space of degree $k$ functions on a graded
manifold~$\calN$.
A section of $T^*M \oplus \bR$,
$\alpha + f \in \Omega^1(M) \oplus C^{\infty}(M)$ is identified to
a degree one function
$\ualpha + f s \in C^{\infty}(T[1]M \oplus \bR[1])$.
The space of functions $C^{\infty}(T[1]M \oplus \bR[1])$ as an extension of
$\Omega^1(M) \oplus C^{\infty}(M)$, i.e.,
$\Omega^1(M) \oplus C^{\infty}(M)$ is a subset of
$C^{\infty}(T[1]M \oplus \bR[1])$.
Thus, \[\Omega^1(M) \oplus C^{\infty}(M) \simeq
C_1^{\infty}(T[1]M \oplus \bR[1]) \subset C^{\infty}(T[1]M \oplus \bR[1]).\]

If $M$ is a Poisson manifold, $T^*[1]M \oplus \bR[1]$ is a dg manifold
(a Q-manifold).
Because a~homological vector field is constructed
from the Poisson structure on $M$ as
\[
Q =\pi^{ij}(x) y_j \frac{\partial}{\partial x^i}
+ \frac{1}{2} \partial_k \pi^{ij}(x) y_i y_j \frac{\partial}{\partial y_k}
+ \frac{1}{2} \pi^{ij}(x) y_i y_j \frac{\partial}{\partial t}.
\]
$Q^2=0$ if $\pi = \frac{1}{2} \pi^{ij}(x) \tfrac{\partial}{\partial x^i}
\wedge \tfrac{\partial}{\partial x^j} \in \Gamma\bigl(\wedge^2 TM\bigr)$
is a Poisson bivector field.

$\calM$ has a dg symplectic structure\footnote{A dg symplectic structure is also called a QP-structure.} as follows.
Since $\calM$ is a cotangent bundle, there exist the following canonical
graded symplectic form of degree two,
\begin{gather}
\omegam = \delta x^i \wedge \delta \xi_i +
\delta \eta^i \wedge \delta y_i + \delta s \wedge \delta t.
\label{TMRsymp}
\end{gather}
We consider the following function of degree $3$:
\begin{align}
\Thetam &=
\pi(\xi, y) - \iota_y (\rd \pi)(y, y) + \pi(y, y)s
\nonumber \\
&=
\pi^{ij}(x) \xi_i y_j - \frac{1}{2} \partial_i \pi^{jk}(x) y_j y_k \eta^i
+ \frac{1}{2} \pi^{jk}(x) y_j y_k s,
\label{homologicalM}
\end{align}
where
$\pi = \frac{1}{2} \pi^{ij}(x) \tfrac{\partial}{\partial x^i}
\wedge \tfrac{\partial}{\partial x^j}$
is a Poisson bivector field on $M$.
The (graded) Poisson bracket of degree $-2$,
$\sbv{-}{-}_{\calM}$, is induced from the symplectic form
\eqref{TMRsymp} as usual.
$\Thetam$ satisfies $\sbv{\Thetam}{\Thetam}_{\calM}=0$ using
$\pi$ is a Poisson bivector field.
Thus we obtain a homological vector field
$Q_{\calM} = \sbv{\Thetam}{-}$ satisfying $Q_{\calM}^2=0$.
$(\omegam, \Thetam)$ define a dg symplectic structure
on $\calM$.
The homological function $\Thetam$ is regarded as
``a Poisson bivector field'' on $T[1]M \oplus \bR[1]$
and define a~(graded) Poisson bracket of degree one
on $T[1]M \oplus \bR[1]$ as follows.
In fact, the following derived bracket defines a bracket
\begin{gather}
\sbv{\uu}{\uv}_{T[1]M \oplus \bR[1]} :=
\sbv{\sbv{\uu}{\Thetam}_{\calM}}{\uv}_{\calM}
\label{gradedPoisson01}
\end{gather}
for $\uu, \uv \in C^{\infty}(T[1]M \oplus \bR[1])$.
We can easily prove that
$\sbv{-}{-}_{T[1]M \oplus \bR[1]}$ is graded skew symmetric,
and satisfies the Leibnitz rule and the Jacobi identity, thus, a (graded) Poisson bracket on~$C^{\infty}(T[1]M \oplus \bR[1])$.

Moreover, if
$\uu, \uv \in C_1^{\infty}(T[1]M \oplus \bR[1])$, i.e.,
$\uu = \ualpha + fs \simeq \alpha + f$ and
$\uv = \ubeta + g s \simeq \beta + g$,
the bracket \eqref{gradedPoisson01} is equivalent to
the Lie bracket \eqref{PoissonTMR},
\[
[\alpha + f, \beta + g]_{T^*M \oplus \bR}
=
\sbv{\ualpha + fs}{\ubeta + g s }_{T[1]M \oplus \bR[1]}.
\]
Thus, the graded manifold $T[1]M \oplus \bR[1]$ is a (graded)
Poisson manifold with the Poisson bracket~\eqref{gradedPoisson01}
including the Lie algebroid structure on $T^*M \oplus \bR$.

Next, we consider $A^*$.
For $\Gamma(A) \oplus C^{\infty}(M)$, a bilinear bracket is induced
from the Lie algebroid structure on $A$.
The bilinear bracket on $A^*$ is concretely defined by
\cite{Blohmann:2023}
\begin{gather}
\{f, g \}_{A^*} := 0,
\label{PoissonA01}
\qquad
\{a, g \}_{A^*} := \rhoa(a) g,
\\
\{a, b \}_{A^*} := [a, b]_A,
\label{PoissonA03}
\end{gather}
where $f, g \in C^{\infty}(M)$ and $a, b \in \Gamma(A)$.
Here an element of $a \in \Gamma(A)$ is regarded as a
linear function $a\colon \Gamma(A^*) \rightarrow \bR$
by the pairing of $A$ and $A^*$.
Since the bracket $\{-, -\}_{A^*}$ satisfies the Jacobi identity,
this fiberwise bracket is called a `Poisson bracket' on $A^*$
in literature, by regarding $a \in \Gamma(A)$ is a linear function
on $A^*$.

More generally, this `Poisson bracket'
is induced from a graded Poisson bracket
on the space of functions on a dg symplectic manifold.
For a vector bundle $A$ over $M$, $A[1]$ is a graded bundle,
of which fiber is shifted by one.
For a Lie algebroid $A$, the graded manifold $A[1]$ is
a~dg manifold~\cite{Vaintrob}.
Let $x^i$ be a local coordinate on the base manifold $M$,
$q^a$ be a fiber coordinate on~$A[1]$.
A~vector field of degree one on the shifted vector bundle
$A[1]$ is given by
\[
Q = \rhoa^i_a(x) q^q \frac{\partial}{\partial x^i}
- \frac{1}{2} C_{ab}^c(x) q^a q^b \frac{\partial}{\partial q^c}.
\]
Require that $A[1]$ is a dg manifold, i.e., $Q$ is a homological vector field
such that $Q^2 = \frac{1}{2} [Q, Q]\allowbreak =0$.
This condition gives a Lie algebroid structure on $A$.
Here the anchor map and the Lie bracket on $A$ are given by
$
\rhoa(e_a) := \rhoa^i_a(x) \partial_i$,
$
[e_a, e_b] := C_{ab}^c e_c$,
for the basis $e_a$ of $A$.

The corresponding dg symplectic manifold is constructed as follows.
We consider the graded cotangent bundle
$\calN = T^*[1](A[1] \oplus \bR[1]) \simeq T^*[1](A^*[1] \oplus \bR[1])$.
We take the canonical graded symplectic form on the cotangent bundle.
Then a Hamiltonian function
$\Thetan \in C^{\infty}(\calN)$ for $Q$ is defined by
$
\delta \Thetan = - \iota_Q \omega_{\calN}$,
where $\delta$ is the differential on the graded manifold $\calN$.
Since $Q^2=0$, $\Theta$ satisfies that
\begin{gather}
\sbv{\Thetan}{\Thetan}_{\calN} =0.
\label{cmq}
\end{gather}
Take local coordinates on $\calN = T^*[1](A^*[1] \oplus \bR[1]) $.
$\bigl(x^i, p_a, s\bigr)$ are local coordinates on $A^*[1] \oplus \bR[1]$ of degree $(0, 1, 1)$
and $(\xi_i, q^a, t)$ are the corresponding fiber coordinates
of degree $(2, 1, 1)$.
The graded symplectic form is
$
\omegac = \delta x^i \wedge \delta \xi_i + \delta p_a \wedge \delta q^a
+ \delta s \wedge \delta t$.
The local coordinate expression of the homological function $\Thetan$ is
\begin{gather}
\Thetan =
\rhoa^i_a(x) \xi_i q^a + \frac{1}{2} C_{ab}^c(x) q^a q^b p_c.
\label{homologicalN}
\end{gather}
Equation~\eqref{cmq} shows that $\Thetan$ is regarded as a
'Poisson bivector field' on $A^*[1] \oplus \bR[1]$.
A Poisson brackets on $A^*[1] \oplus \bR[1]$ is defined by the derived bracket
\begin{gather}
\{F, G \}_{A^*[1] \oplus \bR[1]} := - \sbv{\sbv{F}{\Theta}_{\calN}}{G}_{\calN}
\label{derivedbracket}
\end{gather}
for $F, G \in C^{\infty}(A^*[1] \oplus \bR[1])$.

A degree one function on $A^*[1]$, $\ua = a^a(x) p_a \in C_1^{\infty}(A^*[1])$
is identified to a~section $a = a^a(x) e_a \in \Gamma(A)$.
A degree one function on $M \times \bR[1]$ is $f s \in C_1^{\infty}(M \times \bR[1])$ is identified to a~function $f \in C^{\infty}(M)$.
For degree one functions,
$F = \ua + fs \simeq a + f$ and $G = \ub + gs \simeq b + g$,
the Poisson bracket \eqref{derivedbracket} coincides with
equations~\eqref{PoissonA01}--\eqref{PoissonA03},
$
[a + f, b + g]_{A}
= \{\ua + fs, \ub + gs \}_{A^*[1] \oplus \bR[1]}$,
i.e., $\sbv{a + f}{b + g}_{A^*}= \{\ua + f, \ub + g \}_{A^*[1] \oplus \bR[1]}$.
Therefore, the Poisson bracket for linear functions on $A^*$,
equations~\eqref{PoissonA01}--\eqref{PoissonA03}, is regarded as
the restriction of the Poisson bracket on $C^{\infty}({A^*[1] \oplus \bR[1]})$
equation~\eqref{derivedbracket}.

\subsection{Poisson maps between dg symplectic manifolds}\label{sec:Poissonmap}

In this section, we construct a Poisson map from
$(T[1]M \oplus \bR[1], \{-,-\}_{T[1]M \oplus \bR[1]})$
to $(A^*[1],\allowbreak \{-, -\}_{{A^*[1] \oplus \bR[1]}})$ induced from a momentum section $\mu$.
Obviously, it gives
a Poisson map between two ordinary manifolds, $(T^*M \oplus \bR, \{-,-\}_{T^*M \oplus \bR})$ to $(A^*, \{-, -\}_{A^*})$.

For $\mu \in \Gamma(A^*)$, the covariant derivative
$\nabla \mu \in \Omega^1(M, A^*)$
is a $1$-form taking a value on $A^*$, and
regarded as a map $\nabla \mu\colon T[1]M \rightarrow A^*[1]$.
Thus, we have the map $-\nabla \mu + \mu\colon T[1]M \oplus \bR[1]
\rightarrow A^*[1]$.

\begin{Theorem}\label{ThPoismap}
Let $(A, \pi, \nabla, \mu)$ be a Hamiltonian Lie algebroid over a Poisson manifold $M$.
If~${\bracket{\baS}{\mu} =0}$,
$-\nabla \mu + \mu\colon T[1]M \oplus \bR[1] \rightarrow A^*[1]$
is a Poisson map.
\end{Theorem}

\begin{proof}
The statement of theorem is that the following equation is satisfied:
\[
\sbv{(-(\nabla \mu)^* + \mu^*) F}{(-(\nabla \mu)^* + \mu^*)G}_{T[1]M \oplus \bR[1]} = (-(\nabla \mu)^* + \mu^*)(\sbv{F}{G}_{A^*[1]})
\]
for every $F, G \in C^{\infty}(A^*[1])$.

We can directly prove by calculating this equation
using the derived bracket construction of
Poisson brackets of two spaces $T^*[1](T[0]M \oplus \bR[0])$ and $T^*[2]A^*[1]$
with homological functions, \eqref{homologicalM}, \eqref{gradedPoisson01}, \eqref{homologicalN} and \eqref{derivedbracket}.
Here $F, G \in C^{\infty}(A^*[1])$ are arbitrary functions of $x^i$ and $p_a$,
$F= F(x, p)$ and $G= G(x, p)$.
$-(\nabla \mu)^* + \mu^*$ corresponds to the degree two element $-(\unablamu)^* + \umu^* =
(-\nabla_i \mu_a) \eta^i q^a + \mu_a s q^a$ on the dg manifold.
By direct calculations,
we can prove that if \smash{$\bracket{\baS}{\mu} =0$},
the following equation is satisfied
in the derived bracket:
\begin{align*}
\sbv{\sbv{\sbv{-\unablamu^* + \umu^*}{F}_{\calM}}{\Thetam}_{\calM}}
{\sbv{-\unablamu^* + \umu^*}{G}_{\calM}}_{\calM} &=
\sbv{-\unablamu^* + \umu^*}{\sbv{\sbv{F}{\Thetan}_{\calN}}{G}_{\calN}}_{\calN}.
\end{align*}

Another simpler proof is to use Proposition \ref{dualPoisson}
and the result in Section \ref{PoisLAM}.
$-\nabla \mu + \mu$ is nothing but the dual of
$-(\nabla \mu)^* + \mu^*$.
In Proposition \ref{PoisP3}, we proved that it
$-(\nabla \mu)^* + \mu^*$ is the Lie algebroid morphism
between $A$ to $T^*M \oplus \bR$.
Therefore, Theorem \ref{ThPoismap} is obtained that~${-\nabla \mu + \mu}$ is a Poisson map from
$TM \oplus \bR$ to $A^*$.
\end{proof}

\section{Momentum sections as Dirac morphisms}\label{sec:Dirac}
In previous sections, Lie algebroid morphisms and Poisson maps
of momentum sections have been realized in Dirac structures $L$.
This suggests a generalization of a Poisson map to a~morphism between
Dirac manifolds.
Moreover, we propose Hamiltonian Lie algebroids over Dirac structures.

We introduce a Dirac morphism \cite{AlekseevBursztynMeinrenken} between two Dirac structures. It is a generalization of a~Dirac map \cite{BursztynCrainic1, BursztynCrainic2} and is related to the morphism between Dirac structures with other names~\cite{BursztynIglesiasPonteSevera}.

Let $\varphi\colon M\to N$ be a smooth map between smooth manifolds $M$ and $N$.
We define a binary relation $\mathbb{T}\varphi$ from $\mathbb{T}M:=TM\oplus T^*M$ to $\mathbb{T}N:=TN\oplus T^*N$, denoted by
$
\mathbb{T}\varphi \colon\ \mathbb{T}M \dashrightarrow \mathbb{T}N$,
as follows:
two elements $u+\sigma \in T_mM\oplus T_m^*M$ and $v+\tau \in T_nN\oplus T^*_nN$ for some $m\in M$, $n\in N$ are said to be in the \textit{relation} $\mathbb{T}\varphi$
if
$
n=\varphi(m)$, $ v=(\mathrm{d}\varphi)_m(u)$ and $ \sigma = (\mathrm{d}\varphi)_m^*(\tau)$,
hold. Here, the comorphism $(\mathrm{d}\varphi)_m^*$ for each $m\in M$ defines a relation in the opposite direction, from $T_n^*N$ to~$T_m^*M$ with $n=\varphi(m)$.
We write them as $u+\sigma\mapstochar\dashrightarrow_{\mathbb{T}\varphi} v+\tau$. The relation $\mathbb{T}\varphi$ defines a subset in $\mathbb{T}M\times \mathbb{T}N$.
Note that $(\mathrm{d}\varphi)^*$ is a relation, not a map.

We say that sections $u+\alpha \in \Gamma(\mathbb{T}M),\,v+\beta\in \Gamma(\mathbb{T}N)$ are in the relation $\mathbb{T}\varphi$ which is denote by
$u+\alpha\mapstochar\dashrightarrow_{\mathbb{T}\varphi} v+\beta$ if elements $u_m+\alpha_m\in T_mM\oplus T_m^*M$ and $v_n+\beta_n\in T_nN\oplus T_n^*N$ are in the relation $\mathbb{T}\varphi$ at each $m\in M, n\in N$.

Let $(M,L_M)$ and $(N,L_N)$ be two Dirac manifolds. A Dirac morphism is defined as follows.
\begin{Definition}
Let $\varphi\colon M\to N$ be a smooth map. A Dirac morphism {\rm (}or a forward Dirac map{\rm )} is defined to be a binary relation
$
\mathbb{T}\varphi \colon (\mathbb{T}M,\,L_M)\dashrightarrow (\mathbb{T}N, L_N)
$
satisfying the following property{\rm :} for any $m\in M$ and $v+\beta\in (L_N)_{\varphi(m)}$, there exists a unique element $u+\alpha\in (L_M)_m$ such that
$u+\alpha\mapstochar\dashrightarrow_{\mathbb{T}\varphi} v+\beta$.
\end{Definition}

There exists the following relation of a Poisson map
with a Dirac morphism \cite{Meinrenken02}.
\begin{Proposition}\label{PoissonDirac}
Let $(M, \pi_M)$ and $(N, \pi_N)$ be Poisson manifolds.
Let $\mathrm{Gr}(\pi_M)$ be the graph of the map
$\pi_M\colon T^*M \rightarrow TM$.
Then $\varphi\colon M \rightarrow N$ is a Poisson map if and only if
$\bbT \varphi\colon (TM \oplus T^*M, \mathrm{Gr}(\pi_M))
\dashrightarrow (TN \oplus T^*N, \mathrm{Gr}(\pi_N))$
is a Dirac morphism.
\end{Proposition}
We apply Proposition \ref{PoissonDirac} to our settings.

Let $(M, \pi, A, \mu)$ be a Hamiltonian Lie algebroid over a Poisson manifold.
Remember results in Section \ref{sec:Poissonmap}.
From Theorem \ref{ThPoismap},
$-\nabla \mu + \mu\colon TM \oplus \bR \rightarrow A^*$
is a Poisson map.
We can take $M$ and $N$ in Proposition \ref{PoissonDirac}
as $TM \oplus \bR$ and $A^*$. Then we obtain the following proposition.
\begin{Proposition}\label{Diracmorphism}
Let $(M, \pi, A, \mu)$ be a Hamiltonian Lie algebroid over a Poisson manifold.
Then
$\bbT (-\nabla \mu + \mu)\colon (T(TM \oplus \bR) \oplus T^*(TM \oplus \bR),
\mathrm{Gr}(\pi_{TM \oplus \bR}))
\dashrightarrow (TA^* \oplus T^*A^*, \mathrm{Gr}(\pi_{A^*}))$
is a~Dirac morphism.
\end{Proposition}

\section{Further discussions}
Possible applications and future directions are discussed in this section.

One of important applications of momentum maps is symplectic reductions \cite{MarsdenWeinstein, Meyer} or Poisson reductions \cite{MarsdenRatiu}.
For consistency of reductions, the momentum map must be a Poisson map from~$M$ to the dual of the Lie algebra $\mathfrak{g}^*$.
A momentum section is not necessarily a Poisson map from~$M$ to $A^*$
\cite{Blohmann:2023}.
In our paper, we have constructed a Poisson map from $T^*[1]M \oplus \bR[1]$
to $A^*$ from a momentum section. The map will give a consistent reduction.
One idea to make a~momentum section a Poisson map is to impose a condition compatible with a Poisson structure and a Lie algebroid $A$ \cite{Hirota:2023xqd}.
The idea to construct a Poisson map in this paper is another one.

Hamiltonian Lie algebroids appear in some physical models, constrained mechanics and sigma models \cite{Ikeda:2019pef}. Analysis of Hamiltonian Lie algebroids is directly connected to understanding these physical systems.
In particular, quantizations of these physical models will give a kind of `quantizations' of Hamiltonian Lie algebroids. Precise definition and meaning is not clear. It will give interesting insights to new relations to geometry and quantization.

\subsection*{Acknowledgements}
The authors would like to thank Antonio Miti and referees for useful discussion and comments.
Especially, they would like to thank to the anonymous referees for relevant contribution to improve the paper.
This work was supported by JSPS Grants-in-Aid for Scientific Research Number 22K03323.

\pdfbookmark[1]{References}{ref}
\LastPageEnding

\end{document}